\documentclass[a4paper,reqno]{amsart}
\usepackage{microtype}
\usepackage{amssymb}
\usepackage{amsfonts}
\usepackage{geometry}
\usepackage{hyperref}

\newtheorem{theorem}{Theorem}
\theoremstyle{plain}

\newtheorem{example}{Example}

\newtheorem{lemma}{Lemma}

\numberwithin{equation}{section}
\numberwithin{theorem}{section}  
\numberwithin{proposition}{section}  
\numberwithin{lemma}{section}  
\numberwithin{corollary}{section}

\def \eq{\begin{equation}}
\def \eeq{\end{equation}}

\input epsf

\begin{document}
\title[Body Moving in a Kinetic Sea]{Convergence to equilibrium of a body
moving in a kinetic sea}
\author{Xuwen Chen}
\address{Department of Mathematics, University of Rochester, Rochester, NY
14627}
\email{chenxuwen@math.brown.edu}
\urladdr{http://www.math.rochester.edu/people/faculty/xchen84/}
\author{Walter A. Strauss}
\address{Department of Mathematics and Lefschetz Center for Dynamical
Systems, Brown University, Providence, RI 02912}
\email{wstrauss@math.brown.edu}
\urladdr{http://www.math.brown.edu/\symbol{126}wstrauss/}
\date{v2 for SIMA, 10/24/2015}
\keywords{kinetic, free boundary}
\subjclass[2010]{70F45, 35R35, 35Q83, 70F40}

\begin{abstract}
We consider a continuum of particles that are acted upon by an external
force $\mathbf{G}(t,\mathbf{x})$ and that collide with a rigid body. The
body itself is subject to a constant force $E$ as well as to the collective
force of interaction with the particles. We assume that the particles that
collide with the body reflect probabilistically with some probablility
distribution $K(v,u)$. Under certain conditions on $\mathbf{G}(t,\mathbf{x})$
and $K(v,u)$, we identify an equilibrium velocity $V_{\infty }$ of the body
and we prove that this equilibrium is asymptotically stable.
\end{abstract}

\maketitle

\section{Introduction}

In previous work of this type the particles have been subject to no force ($%
\mathbf{G}(t,\mathbf{x})\equiv 0$) and therefore travel with constant
velocities between their collisions with the rigid body. In this paper we
initiate an investigation of how some forces on the particles may affect the
motion of the body. What is the asymptotic effect of such particle forces on
the body? Our ultimate goal for the future is to treat a plasma in which the
particles are charged and thus are subject to electromagnetic forces, which
may be either external or created by the particles themselves. In this paper
the forces are external.

Our problem has a free boundary, the moving location of the body. The other
unknowns are the configuration and motion of the particles. The particles
may collide with the body diffusely. Boundary interactions in kinetic theory
are very poorly understood, even when the boundaries are fixed. Free
boundaries are even more difficult. Our most important assumption is that
the whole system, consisting of the body and the particles, starts out
rather close to an equilibrium state.

We consider classical particles that are extremely numerous and subject to
an outside force $\mathbf{G}$, which in this introductory study we assume to
be small and decaying. While one could consider modeling the particles as a
fluid, we instead model them as a continuum with a phase-space density $f$
as in kinetic theory \cite{Glassey, Sone, Spohn}. Furthermore, the
interaction of the particles with the body at its boundary can be quite
complicated in typical physical scenarios. For instance, the boundary may be
so rough that a particle may reflect from it in an essentially random way.
There could even be some kind of physical or chemical reaction between the
particle and the molecules of the body. Therefore we model the collisions by
a probability distribution $K$. This is sometimes called diffuse reflection.
Altogether, the prescribed data consists of the initial density $f_0$, the
force $\mathbf{G}$, the collision distribution $K$, and a constant force $E$
on the body.

The present paper is a sequel to \cite{CS-1, CS-2} and is also highly
motivated by a series of remarkable papers \cite{Ital3, Ital2, Ital1}. Most
of the earlier papers, including \cite{Ital3, Ital2, Ital1, CS-1, CS-2},
were devoted to proving more detailed behavior such as determining the
precise rate of approach to equilibrium. In particular, the earliest papers 
\cite{Ital1, Ital2} assumed specular reflection at the free boundary. On the
other hand, considering a perfect gas, the authors of \cite{Ital3} chose a
Maxwellian-like distribution for the reflected particles. More generally,
for the reasons mentioned above we considered in \cite{CS-1, CS-2} as well
as in this paper a general probabilistic law of reflection. In \cite{CS-1,
CS-2}, where $\mathbf{G}=0$, we found a condition on $K$ and $f_{0}$ that is
almost necessary and sufficient that the approach to equilibrium is
reversed; that is, the approach begins from the right side, say, and
ultimately approaches from the left. The present paper does not deal with
such detailed behavior, focusing instead on whether it approaches
equilibrium at all. Some numerical investigations, mostly dealing with a
diffuse reflection which ejects Maxwellians, have appeared in \cite{ATC, TA,
TA2, TAK}. Some other related investigations that deal with the speed of
approach to equilibrium in the specular reflection case are \cite%
{AG,Cav,CM2, Sisti}. See also the survey in \cite{Butta}. In all of the
preceding references, the particles are identical, the body is initially
moving near the equilibrium velocity $V_{\infty }$ and the motion of the
body is one-dimensional, even though the particles move in three (or $d$)
dimensions. The current paper is the first one for which the particles are
themselves subject to a force $\mathbf{G}\neq \mathbf{0}$.

We now describe our problem in mathematical terms. For simplicity we take
the body to be a vertical disk $D(t)$ centered at the point $\left(
X(t),0,0\right) $ and traveling with velocity $\left( V(t),0,0\right) $. We
assume there is a constant horizontal force $E\in \mathbb{R}$ acting on the
body, in addition to the horizontal frictional force $F(t)$ due to all the
colliding particles at time $t$. Thus 
\begin{equation*}
\frac{dX}{dt}=V(t),\quad \frac{dV}{dt}=E-F(t).
\end{equation*}%
The particle distribution $f(t,\mathbf{x},\mathbf{v})$ satisfies the Vlasov
equation 
\begin{equation}
\partial _{t}f+\mathbf{v}\cdot \nabla _{\mathbf{x}}f+\mathbf{G}(t,\mathbf{x}%
)\cdot \nabla _{\mathbf{v}}f=0  \label{pde:real}
\end{equation}%
in $\mathbb{R}^{3}\diagdown D(t)$. We assume the initial velocity $f(0,%
\mathbf{x},\mathbf{v})=f_{0}(\mathbf{v})$ depends only on $\mathbf{v}$ and
is even. At each collision with the body we assume the diffuse law of
reflection 
\begin{equation}
f_{+}(t,\mathbf{x},\mathbf{v})=\int_{(u_{1}-V(t))(v_{1}-V(t))\leq 0}K\left(
v_{1}-V(t);u_{1}-V(t)\right) \ f_{-}(t,\mathbf{x},u_{1},v_{\perp })\ du_{1}
\label{EQ:BC}
\end{equation}%
for $\mathbf{x\in }D(t)$, where the pre/postcollision distributions $f_{\pm
} $ are defined as%
\begin{equation*}
f_{\pm }(t,\mathbf{x},v)=\lim_{\varepsilon \rightarrow 0^{+}}f(t\pm
\varepsilon ,\mathbf{x}\pm \varepsilon \mathbf{v},\mathbf{v}).
\end{equation*}%
In \eqref{EQ:BC} we regard $\mathbf{u}$ as the velocity of a particle coming
into collision and $\mathbf{v}$ as an ejected velocity. We assume that the
collision kernel $K$ conserves mass during collisions. The form of (\ref%
{EQ:BC}) implies that there is no energy exchange between the disk and the
particles except in the horizontal direction. This is natural because the
disk is fixed to move only along the $x$-axis. More detailed assumptions on $%
\mathbf{G},f_{0}$ and $K$ are provided in Section \ref{sec:set up}. We note
that the $F(t)$ notation for the frictional force does not reflect the
complicated nature of the interaction. In fact, due to the recollisions
between the body and the particles, we will see in Lemma \ref{IntRep} that
the frictional force at time $t$ depends on the behavior of the whole
system, including both the disk and the particles, at all previous times.

It is far from obvious what the ultimate velocity $V_{\infty }$ will be. We
prove that $V_{\infty }$ is entirely determined by imagining a situation in
which both $\mathbf{G}=\mathbf{0}$ and there are no precollisions; that is,
the collisions at a time $t$ come from particles moving at constant
velocities that have not previously collided with the body. The limiting
velocity $V_{\infty }$ is determined as follows. We define a
\textquotedblleft fictitious" force $F_{00}$ for which there is no external
field $\mathbf{G}$ at all and the particles do not hit the boundary before
time $t$, namely, 
\begin{equation}
F_{00}(V)=\int_{D(t)}dS_{\mathbf{x}}\int_{\mathbb{R}^{3}}\tilde{L}%
(v_{1}-V(t))\ f_{0}(\mathbf{v})\ d\mathbf{v},  \label{def:F_00}
\end{equation}%
where $\tilde{L}$ is defined in (\ref{functionL}). It is a strictly monotone
function of $V$. We then define $V_{\infty }$ by the equation 
\begin{equation*}
F_{00}(V_{\infty })=E.
\end{equation*}%
%
%
%
%We remark that this $V_{\infty }$ happens to be the
%same one in the $\mathbf{G=}0$, no external field case and it is interesting that they are identical.  
%Although it is not obvious that the equilibrium velocity $V_{\infty }$ does
%not take account of either the force \textbf{G} or the precollisions, it turns out to be the case.
%Thus we shall prove, although it is not obvious \textit{a priori}, that the
%asymptotic velocity is independent of both the force $\mathbf{G}$ and the precollisions. 
%The equilibrium velocity $V_{\infty }$ takes account of neither 
%the particle force $\bG$ because it decays in time,  
%nor the recollisions because they only affect the rate of decay and not the limit. 

Our main result is stated informally as follows. If the initial velocity $%
V(0)$ is sufficiently close to $V_{\infty }$, then there exists a solution $%
(V(t),f(t,\mathbf{x},\mathbf{v}))$ of our problem such that $V(t)\rightarrow
V_{\infty }$ as $t\rightarrow \infty $. The same asymptotic statement is
valid for any solution of the problem. This result is stated with precision
and further details in Section \ref{sec:set up}. The present result is the
first one for which the particles are themselves subject to a force $\mathbf{%
G}\neq \mathbf{0}$, so that between collisions with the body the individual
particles travel in curves rather than straight lines. As a consequence, we
have limited our investigation to proving the convergence (at some rate) to
the equilibrium velocity $V_{\infty }$.

We state the basic iteration scheme to obtain the main theorem in %
\eqref{approxeqV}. The core of the proof is a series of technical estimates.
We emphasize that adding an external force on the continuum of particles,
even considering only the rather ideal situation where the external force $%
\mathbf{G}(t,\mathbf{x})$ acting on them is small and decaying, makes the
problem much more difficult to deal with. Because we are interested in the
long-time behavior of the body, even a small force $\mathbf{G}$ could have
significant effects for large time which we explain below.

Relating the aforementioned $V_{\infty }$, which comes from an imaginary 
case in which there are no collisions and $\mathbf{G=0}$, to our system (for 
which there are collisions and $\mathbf{G\neq 0}$) is the most delicate and 
novel part of the analysis. 
In Section \ref{Sec:CharacteristicEstimates} we estimate the force $H(t)$ on the body at 
a large time $t$ due to $\mathbf{G}$ in the fictitious situation that there 
have been no precollisions. $%
H(t)$ is given by an integral over the various velocities $v$ of the
incoming particles and we prove that $H(t)$ is integrable over $0<t<\infty $%
. We estimate the integral in five pieces, using an intermediate time $0<T<t$%
. One piece $H_{1}$ is small (for large $t$) because $v-V_{\infty }$ is
small. The second piece $H_{21}$ is proven small by using the decay of $%
\mathbf{G}$ in time. The third piece $H_{221}$ makes use of the decay of $%
\mathbf{G}$ in both space and time. The fourth piece $H_{222}$ makes use of
the decay of $\mathbf{G}$ in space. The fifth piece $H_{223}$ particularly
makes use of the decay of $f_{0}$ as well as the decay of $\mathbf{G}$ in
time. There is a range of possible choices for the parameters $p,m,q$,
summarized in Theorem \ref{Prop:H is small}.

Furthermore, because the body is moving and the particles travel in curves, it
is evidently more difficult to track the collisions between the particles
and the body, in contrast to the $\mathbf{G}=\mathbf{0}$ case for which the
particles travel in straight lines.  A given particle can collide
with the body many times or even an infinite number of times. This is
partially discussed in Lemma \ref{IntRep} and at the beginning of Section %
\ref{Sec:recollisions}. Nevertheless, the effect of the precollisions can be
estimated by the most recent precollision. 
This delicate part of the proof is given in Lemma \ref{R_W estimate}.  
Finally, putting the different pieces
together is a matter of dealing with several competing small quantities, as
we discuss in Section \ref{Sec:existence}, where we combine the previous
estimates and prove the required convergence.

\subsection{Acknowledgements}

The authors would like to thank Kazuo Aoki once again for introducing them
to this general subject and Yan Guo for discussions about multiple
collisions. This research was supported in part by NSF grant DMS-1464869.

%%%%%%%%%%%%%%%%%%%%%%%%%

\section{Assumptions and preliminaries\label{sec:set up}}

We write the coordinates as $\mathbf{x}=(x_{1},x_{\perp
})=(x_{1},x_{2},x_{3}),\mathbf{v}=(v_{1},v_{\perp })=(v_{1},v_{2},v_{3}),%
\mathbf{G}=(G_{1},G_{\perp })=(G_{1},G_{2},G_{3})$. We begin by listing the
assumptions precisely. We assume that the force $\mathbf{G}$ has the form 
\begin{equation}
\mathbf{G}(t,\mathbf{x})=(G_{1}(t,x_{1}),G_{\perp }(t,x_{\perp }))
\label{eq:form of G}
\end{equation}%
and that it decays like 
\begin{equation}
\left\vert G_{1}(t,x_{1})\right\vert \leqslant \frac{c_{G}}{\left\langle
t\right\rangle ^{q}\left\langle x_{1}\right\rangle ^{m}}\text{ and }%
\left\vert G_{\perp }(t,x_{\perp })\right\vert \leqslant \frac{c_{G}}{%
\left\langle t\right\rangle ^{q}}  \label{eq:decay of G}
\end{equation}%
where $c_{G}$ is a constant, $q>2$ and $m>0$. The structural assumption (\ref%
{eq:form of G}) is crucial to our proof although we hope to weaken it in
future work. In the main theorem below we obtain a decay rate that does not
depend on the dimension because we refrain from making an even stronger
assumption. In fact, even with our decay assumption (\ref{eq:decay of G}) on 
$G_{\perp }$, some of the particles might collide with the disk many times
and might not even escape the disk at all.

We assume that the initial velocity distribution $f_{0}$ is a nonnegative
even $C^{1}$ function $\not\equiv 0$ that has the form 
\begin{equation}
f_{0}(\mathbf{v})=a_{0}(v_{1})\ b_{0}(v_{\perp })  \label{eq:form of f_0}
\end{equation}%
and that it decays like 
\begin{equation}
\left\vert \nabla _{\mathbf{v}}f_{0}(\mathbf{v})\right\vert \leqslant \frac{C%
}{\left\langle v_{1}\right\rangle ^{l_{1}+1}\left\langle v_{\perp
}\right\rangle ^{l_{2}+1}}  \label{decay:initial}
\end{equation}%
where $l_{1}>q+1$ and $l_{2}>1$. (We remark that it is sufficient to assume
the slightly weaker condition 
\begin{equation*}
\int_{\mathbb{R}}\left\vert \frac{\partial a_{0}}{\partial v_{1}}\right\vert
\langle v_{1}\rangle ^{q}dv_{1}<\infty ,\quad \int_{{\mathbb{R}}^{2}}|\nabla
b_{0}(v_{\perp })|dv_{\perp }<\infty
\end{equation*}%
although we will not bother to explicitly use this condition.)

%\begin{condition} 
Define $\gamma =|V(0)-V_{\infty }|$. This quantity is used throughout the
paper and will be assumed to be very small. We assume that $K(v_{1},u_{1})$
is a continuous nonnegative function, $C^{1}$ for $u_{1}\neq 0,v_{1}\neq 0$,
that is even in both variables separately ($%
K(v_{1},u_{1})=K(-v_{1},u_{1})=K(v_{1},-u_{1})$) and that is bounded for $%
v_{1}\in {\mathbb{R}},\left\vert u_{1}\right\vert \leq 3\gamma $. We impose
the mass conservation condition 
\begin{equation}
\int_{0}^{\infty }v_{1}K(v_{1},u_{1})dv_{1}=|u_{1}|,  \label{Kmass}
\end{equation}%
the power law condition 
\begin{equation}
\sup_{|u_{1}|<\gamma +c_{G}}\int_{0}^{\infty
}v_{1}^{2}K(v_{1},u_{1})dv_{1}\leq C|u_{1}|^{p}  \label{Kestimate}
\end{equation}%
with $0<p\leq 2$, and the integrability condition 
\begin{equation}
K(v_{1},z-y-V_{\infty })\ \left\langle z\right\rangle ^{-l_{1}}\leq
M(z)\quad \text{ for }|v_{1}|<3\gamma ,\ |y|<3\gamma ,\ |z|<\infty
\label{KIntegral}
\end{equation}%
where $M\in L^{1}(\mathbb{R})$. We also define 
\begin{equation}
L(u_{1})=u_{1}^{2}+\int_{{\mathbb{R}}}v_{1}^{2}K(v_{1},u_{1})dv_{1}\quad 
\text{ and }\quad \tilde{L}(u_{1})={\text{sgn}}(u_{1})L(u_{1}).
\label{functionL}
\end{equation}%
Because $K$ is even, $L(u_{1})$, as defined, is also even. In addition, we
assume that the even function $L(u_{1})$ is decreasing for $u_{1}<0$. Then
we define the collision operator as $\mathcal{K}_{t}(f_{-})=f_{+}$ as in %
\eqref{EQ:BC}. %\end{condition}

To summarize, the assumptions listed above have five exponents $p,m, q,
l_1,l_2$ that satisfy $0<p\le 2,\ q>2,\ m>0,\ l_1>q+1,\ l_2>1$. We define
another parameter $\sigma$ by 
\begin{equation}  \label{sigma}
\frac{1}{\sigma }=\frac{1}{p+1}+\frac{1}\mu, \qquad \mu = \min (m,q-1).
\end{equation}
We also have the two parameters $\gamma$ and $c_G$ in (\ref{eq:decay of G})
that will be chosen sufficiently small. 
%%%%%%%%%%%%%%%%%   THEOREM   %%%%%%%%%%%%%%%%%%
Our main theorem is as follows.

\begin{theorem}
\label{ThExistence} Let $\mu >1+\frac{1}{p}$. There is a constant $A$ such
that if $\gamma =|V(0)-V_{\infty }|$ and $c_{G}$ are sufficiently small,
then there exists a solution $(V(t),f(t,\mathbf{x},\mathbf{v}))$ of our
problem for which $V\in C^{1}([0,\infty ))$, $f\in L^{\infty }([0,\infty
)\times {\mathbb{R}}^{3}\times {\mathbb{R}}^{3})$ and 
\begin{equation}
\left\vert V(t)-V_{\infty }\right\vert \leqslant \gamma e^{-b_{0}t}+\frac{%
A\gamma ^{p+1}}{(1+t)^{\sigma }}  \label{V decay}
\end{equation}%
where $b_{0}=\min_{V\in \left[ V_{\infty }-3\gamma ,V_{\infty }+3\gamma %
\right] }F_{00}^{\prime }(V)$. The pair of functions $f_{\pm }(t,\mathbf{x},%
\mathbf{v})$ are a.e. defined explicitly in terms of $V(t)$ and $f_{0}(%
\mathbf{v})$. Furthermore, any solution of the problem (in the sense stated
above) satisfies \eqref{V decay}.
\end{theorem}

We now mention three examples of collision kernels for which the reflected
velocity distribution is Gaussian for each incoming particle.

\begin{example}
Let 
\begin{equation*}
K(v_{1},u_{1})=2\beta e^{-\beta v_{1}^{2}}\left\vert u_{1}\right\vert .
\end{equation*}%
As shown in \cite{CS-1,CS-2}, $K(v_{1},u_{1})$ satisfies the assumptions
with $p=1$. In this case, we require $q>3$ and $m>2$ in (\ref{eq:decay of G}%
), and$\ l_{1}>4$ in (\ref{decay:initial}).
\end{example}

\begin{example}
We now choose 
\begin{equation*}
K(v_{1},u_{1})=2e^{-\frac{v_{1}^{2}}{\left\vert u_{1}\right\vert }}.
\end{equation*}%
It is proved in \cite{CS-1,CS-2} that this kernel checks our assumptions
with $p=\frac{3}{2}$. In this case, we need $q>8/3$ and$\ m>5/3$ in (\ref%
{eq:decay of G}), and$\ l_{1}>11/2$ in (\ref{decay:initial}). In this
example an incoming particle with almost the same velocity as the body is
likely to be reflected with almost the same velocity, while an incoming
particle with a very different velocity is reflected according to a very
wide Gaussian around $V(t)$.
\end{example}

\begin{example}
We now scale the last example by setting 
\begin{equation*}
K(v_{1}\mathbf{,}u_{1})=C\left\vert u_{1}\right\vert ^{\beta }e^{-{v_{1}^{2}}%
{\left\vert u_{1}\right\vert ^{\beta -1}}}
\end{equation*}%
where $\beta \in \left[ -1,3\right) $ and $C$ is chosen such that mass is
conserved, that is, (\ref{Kmass}) is satisfied. The most important feature
of this kernel is that, as $\beta $ runs through $\left[ -1,3\right) $, $p$
runs through $\left( 0,{2}\right] $. At the endpoints of this interval we
have the following behavior. If $p=2$, then $q>5/2$ and$\ m>3/2$ in (\ref%
{eq:decay of G})$,\ l_{1}>7/2$ in (\ref{decay:initial}). If $p\rightarrow 0$%
, then $q,m\rightarrow \infty $ in (\ref{eq:decay of G})$,\ l_{1}\rightarrow
\infty $ in (\ref{decay:initial}).
\end{example}

\bigskip We proceed by pointing out some basic properties of our model. The 
\textit{total mass} $M=\int d\mathbf{x}\int d\mathbf{v}f(t,\mathbf{x},%
\mathbf{v})$ of the particles is an invariant. Indeed, it is obviously
invariant under the flow \eqref{pde:real} while it is also preserved under
collisions due to \eqref{Kmass}, as proven in \cite[Lemma 2.3]{CS-1}. The 
\textit{total horizontal force} at time $t$ due to the collisions is 
\begin{equation}
F(t)=\int_{D(t)}dS_{\mathbf{x}}\int_{{\mathbb{R}}^{3}}d\mathbf{v}\ \tilde{L}%
(v_{1}-V(t))\ f_{-}(t,\mathbf{x},\mathbf{v})
\end{equation}%
as shown in \cite[Lemma 2.2]{CS-1}. It is the sum of the force due to
particles on the right ($v_{1}<V(t)$) and the left ($v_{1}>V(t)$).

%%%%%%%%%%%%   CLASS W    %%%%%%%%%%%%%%%%%%
Our basic technique is to look for \textit{possible velocities} in the class 
$\mathcal{W}$, where we define $W\in \mathcal{W}$ if and only if $W(t)$ is
Lipschitz, $W(0)=V_{0}$, and 
\begin{equation}
\left\vert W(t)-V_{\infty }\right\vert \leqslant \gamma e^{-b_{0}t}+\frac{%
A\gamma ^{p+1}}{(1+t)^{\sigma }}  \label{def:W}
\end{equation}%
where the constant $A$ will be specified later in Lemma \ref{coro:V_WinW}
and the power $\sigma $ is given by \eqref{sigma}. Among other conditions to
be stated later, we require $A>1$ and $\gamma <<1$ so that $A\gamma <1$. For 
$W\in \mathcal{W}$ we define $X(t)$ as the primitive of $W(t)$: 
\begin{equation*}
\frac{dX}{dt}=W,\quad X(0)=0.
\end{equation*}%
Until the last section of this paper we will fix any possible velocity $W\in 
\mathcal{W}$.

We introduce the \textit{characteristics} of the PDE \eqref{pde:real} 
\textit{without collisions} as follows. 
%\begin{definition}    \label{Def:S}  
Given $(t,\mathbf{x},\mathbf{v})$, define $\left( \mathbf{\check{x}}(s;t,%
\mathbf{x},\mathbf{v}),\mathbf{\check{v}}(s;t,\mathbf{x},\mathbf{v})\right) $
to be the solution to the characteristic equations 
\begin{eqnarray}
\frac{d\mathbf{\check{x}}(s;t,\mathbf{x},\mathbf{v})}{ds} &=&\mathbf{\check{v%
}}(s;t,\mathbf{x},\mathbf{v})  \label{eqn:characteristic} \\
\frac{d\mathbf{\check{v}}(s;t,\mathbf{x},\mathbf{v})}{ds} &=&\mathbf{G}(s,%
\mathbf{\check{x}}(s;t,\mathbf{x},\mathbf{v}))  \notag
\end{eqnarray}%
with final condition $\left( \mathbf{\check{x}}(t;t,\mathbf{x},\mathbf{v}),%
\mathbf{\check{v}}(t;t,\mathbf{x},\mathbf{v})\right) =\left( \mathbf{x},%
\mathbf{v}\right) $. In particular, for $\mathbf{G}(t,\mathbf{x})$ of the
separated form (\ref{eq:form of G}), the horizontal and vertical components
are decoupled: 
\begin{eqnarray}
(\check{x}_{1}(r;t,\mathbf{x},\mathbf{v}),\check{v_{1}}(r;t,\mathbf{x},%
\mathbf{v})) &=&(\check{x}_{1}(r;t,x,v_{1}),\check{v}_{1}(r;t,x,v_{1}))
\label{eq:simplified characteristic} \\
(\check{x}_{\perp }(r;t,\mathbf{x},\mathbf{v}),\check{v}_{\perp }(r;t,%
\mathbf{x},\mathbf{v})) &=&(\check{x}_{\perp }(r;t,x_{\perp },v_{\perp }),%
\check{v}_{\perp }(r;t,x_{\perp },v_{\perp })).  \notag
\end{eqnarray}

For any function $\phi $ on phase space, we define the \textit{solution
operator} $\mathcal{S}_{tr}$ as 
\begin{equation*}
(\mathcal{S}_{tr}\phi )(\mathbf{x},\mathbf{v})=\phi (\mathbf{\check{x}}(r;t,%
\mathbf{x},\mathbf{v}),\mathbf{\check{v}}(r;t,\mathbf{x},\mathbf{v})).
\end{equation*}%
Then of course $f_{NB}=\mathcal{S}_{t0}f_{0}$ solves the Vlasov equation 
\begin{eqnarray}
\left( \partial _{t}+\mathbf{v}\cdot \nabla _{\mathbf{x}}+\mathbf{G}(t,%
\mathbf{x})\cdot \nabla _{\mathbf{v}}\right) f_{NB} &=&0  \label{pde:free} \\
f_{NB}(0,\mathbf{x,v}) &=&f_{0}(\mathbf{v})  \notag
\end{eqnarray}%
in $\mathbb{R}^{3}$ with no boundary conditions. %\end{definition}

%%%%%%%%%%  Fictitious  %%%%%%%%%%%%%
The fictitious force $F_{00}$, previously defined in (\ref{def:F_00}), is a
strictly monotone function of $V$, due to $L$ being an even function that is
decreasing for $u_{1}<0$. For completeness we state this as a lemma. As a
consequence, $V_{\infty }$ will be uniquely defined by $F_{00}(V_{\infty
})=E $.

\begin{lemma}
\label{Lem:F_00}Suppose $f_{0}(\mathbf{v})\geqslant 0$\ is even, continuous
and $\not\equiv 0$. If $L\in C^{1}$ is an even function with $L(0)=0$, and $%
L^{\prime }(u_{1})<0$ for $u_{1}\in (-\infty ,0)$, then $F_{00}(V)$ is an
increasing odd $C^{1}$ function of $V$.
\end{lemma}

\begin{proof}
A simple change of variable yields 
\begin{eqnarray*}
F_{00}(V) &=&C\left( \int_{v_{1}\leqslant V}L(v_{1}-V)f_{0}(\mathbf{v})d%
\mathbf{v}-\int_{v_{1}\geqslant V}L(v_{1}-V)f_{0}(\mathbf{v})d\mathbf{v}%
\right) \\
&=&C\left( \int_{v_{1}\leqslant V}L(v_{1}-V)f_{0}(\mathbf{v})d\mathbf{v}%
-\int_{v_{1}\leqslant -V}L(v_{1}+V)f_{0}(\mathbf{v})d\mathbf{v}\right)
\end{eqnarray*}%
where $C$ is a fixed constant. Thus $F_{00}$ is odd and its derivative is 
\begin{equation*}
F_{00}^{\prime }(V)=-C\left( \int_{v_{1}\leqslant V}L^{\prime
}(v_{1}-V)f_{0}(\mathbf{v})d\mathbf{v}+\int_{v_{x}\leqslant -V}L^{\prime
}(v_{1}+V)f_{0}(\mathbf{v})d\mathbf{v}\right) >0.
\end{equation*}
\end{proof}

Given a time $t$, we will have to estimate the frictional force due to the
collisions of the particles. To that end, we define another
\textquotedblleft fictitious" force as follows (though not as "fictitious"
as $F_{00}$). The \textit{fictitious force} $F_{0}(t)$ for which no particle
hits the boundary before time $t$ is 
\begin{equation}
F_{0}(t)=F_{0}\left( t,W(t)\right) =\int_{D(t)}dS_{\mathbf{x}}\int_{\mathbb{R%
}^{d}}\tilde{L}(v_{1}-W(t))\ f_{NB}(t,\mathbf{x},\mathbf{v})\ d\mathbf{v}
\label{F0}
\end{equation}%
where $f_{NB}(t,\mathbf{x},\mathbf{v})$ solves the kinetic equation (\ref%
{pde:free}) with no boundary condition and is subject to the initial
condition $f_{NB}(0,\mathbf{x},\mathbf{v})=f_{0}(\mathbf{v})$. As before,
the fictitious force $F_{00}$ is defined so that there is no external field $%
\mathbf{G}$ at all \textit{and} the particles do not hit the boundary before
time $t$, namely, 
\begin{equation}
F_{00}=F_{00}\left( W(t)\right) =\int_{D(t)}dS_{\mathbf{x}}\int_{\mathbb{R}%
^{d}}\tilde{L}(v_{1}-W(t))\ f_{0}(\mathbf{v})\ d\mathbf{v}.  \label{F00}
\end{equation}%
%
%
%
%
%
%
%
%
%
%
%
%%%%%%%%%%%%%%  ITERATION SCHEME  %%%%%%%%%%%%%%
The \textit{force due to the precollisions} on the body moving at velocity $%
W(t)$ then can be expressed as 
\begin{eqnarray*}
R_{W}(t) &=&F(t,W(t))-F_{0}(W(t)) \\
&=&\int_{D(t)}\int_{v_{1}\geqslant W(t)}L(v_{1}-W(t))\left[ f_{NB}(t,\mathbf{%
x},\mathbf{v})-f_{-}(t,\mathbf{x},\mathbf{v})\right] d\mathbf{v}dS_{\mathbf{x%
}} \\
&&+\int_{D(t)}\int_{v_{1}\leqslant W(t)}L(v_{1}-W(t))\left[ f_{-}(t,\mathbf{x%
},\mathbf{v})-f_{NB}(t,\mathbf{x},\mathbf{v})\right] d\mathbf{v}dS_{\mathbf{x%
}} \\
&\equiv &R_{W}^{L}(t)+R_{W}^{R}(t)
\end{eqnarray*}%
comprised of the forces on the right and left sides.

Given any $W\in \mathcal{W}$, we define the function $V_{W}(t)$ by the 
\textit{iteration scheme} 
\begin{equation}
\frac{dV_{W}}{dt}=\frac{E-F_{00}(W)}{V_{\infty }-W}(V_{\infty
}-V_{W})-R_{W}(t)+F_{00}(W)-F_{0}(t,W),\ \ V_{W}(0)=V_{0}.  \label{approxeqV}
\end{equation}%
Note that a fixed point (that is, $V_{W}=W=V$) would satisfy 
\begin{equation*}
\frac{dV}{dt}=\frac{E-F_{00}(V)}{V_{\infty }-V}(V_{\infty
}-V)-R_{V}(t)+F_{00}(V)-F_{0}(t,V)=E-R_{V}(t)-F_{0}(t,V)
\end{equation*}%
and $dX/dt=V$. Thus any fixed point solves our problem. We will estimate the
term $F_{00}(W)-F_{0}(t,W)$ in Section \ref{Sec:CharacteristicEstimates} and
estimate the term $R_{W}(t)$ in Section \ref{Sec:recollisions}.

%%%%%%%%%%%%   MANY COLLISIONS   %%%%%%%%%%%%%%%%
A key difficulty in our problem is that a typical particle may collide with
the body many, or even infinitely many, times. The following lemma, written
for convenience in the 1D case, illustrates this difficulty. It explains how
one can represent the solution by an expansion in terms of a finite number $%
k $ of previous collisions. This expansion might never reach the initial
datum if there are infinitely many collisions; it could happen in many ways.

\begin{lemma}[Integral Representation]
\label{IntRep} Let $\left( x,v\right) \notin Z(t)$ where $Z(t)$ is the set
of measure zero defined in the beginning of Section \ref{Sec:recollisions}.
For brevity we write $\mathcal{B}_{tr}=\mathcal{S}_{tr}\mathcal{K}_{t}$ and $%
J(s,x,v)=-G(s,x)f_{0}^{\prime }(v)$. Given $t>0$, we denote the particle
collision times as $t>t_{1}>t_{2}>...$. 
%by $t_{1}>t_{2}>t_{3}...$ where $t_{1}<t.$ 
Then for arbitrary $k\geqslant 2$ we can represent the solution to $(\ref%
{pde:real})$ subject to the collision law (\ref{EQ:BC}) by 
\begin{eqnarray}
f(t) &=&\sum_{j=1}^{k-1}\left\{ \mathcal{B}_{tt_{1}}...\mathcal{B}%
_{t_{j-1}t_{j}}\int_{\max (t_{j+1},0)}^{t_{j}}\mathcal{S}_{t_{j}s}J(s)ds%
\right\}  \label{f:multiple collision} \\
&&+\chi _{\left\{ 0<t_{k}\right\} }\mathcal{B}_{tt_{1}}...\mathcal{B}%
_{t_{k-1}t_{k}}f_{-}(t_{k})  \notag \\
&&+\int_{\max (t_{1},0)}^{t}\mathcal{S}_{ts}J(s)ds.  \notag
\end{eqnarray}
\end{lemma}

\begin{proof}
It is important to note that the actual number of precollisions could be
larger than $k$ in formula (\ref{f:multiple collision}). Thus if there are
an infinite number of collisions, $f_{0}$ would never show up in (\ref%
{f:multiple collision}). %no matter how large $k$ is.Since this lemma is.  
Note also that each $t_{j}$ depends on $t,x,v$ as well as on the
pre/postcollision velocities at all $t_{i}$ for $1\leqslant i<j$. Because
the lemma is similar to \cite[Lemma 24]{Yan}, we only sketch the proof here.
It is instructive to first write out some of the terms in formula (\ref%
{f:multiple collision}) in total detail including all the variables. For
example, we have 
\begin{equation*}
\mathcal{B}_{tt_{1}}g(x,v)=\mathcal{S}_{t,t_{1}}\mathcal{K}g(x,v)=\int
du_{1}\ K(\check{v}(t_{1}^{+};t,x,v)-V(t_{1}),u_{1}-V(t_{1}))\
g(X(t_{1}),u_{1})
\end{equation*}%
and 
\begin{eqnarray*}
&&\mathcal{B}_{tt_{1}}\mathcal{B}_{t_{1}t_{2}}g(x,v)=\int du_{1}\text{ }K(%
\check{v}(t_{1}^{+};t,x,v)-V(t_{1}),u_{1}-V(t_{1})) \\
&&\times \int du_{2}\ K(\check{v}%
(t_{2}^{+};t_{1},X(t_{1}),u_{1})-V(t_{2}),u_{2}-V(t_{2}))g(X(t_{2}),u_{2})
\end{eqnarray*}%
where $t_{i}^{+}$ means $\lim_{t\rightarrow t_{i}^{+}}$, that is,
postcollision. In fact, exhibiting all the variables, we have 
\begin{eqnarray*}
&&\mathcal{B}_{tt_{1}}\dots \mathcal{B}_{t_{j-1}t_{j}}g(x,v)=\int du_{1}\ K(%
\check{v}(t_{1}^{+};t,x,v)-V(t_{1}),u_{1}-V(t_{1}))\times\dots \\
&&\times \int du_{j}\ K(\check{v}%
(t_{j}^{+};t_{j-1},X(t_{j-1}),u_{j-1})-V(t_{j}),u_{j}-V(t_{j}))\
g(X(t_{j}),u_{j}).
\end{eqnarray*}

We now begin the formal proof with the standard Duhamel formula 
\begin{equation*}
f(t)=\chi _{\left\{ 0<t_{1}\right\} }\mathcal{S}_{tt_{1}}\mathcal{K}%
_{t_{1}}f_{-}(t_{1})+\int_{\max (t_{1},0)}^{t}\mathcal{S}_{ts}J(s)ds.
\end{equation*}%
We then write the next Duhamel formula 
\begin{equation*}
f_{-}(t_{1})=\chi _{\left\{ 0<t_{2}\right\} }\mathcal{S}_{t_{1}t_{2}}%
\mathcal{K}_{t_{2}}f_{-}(t_{2})+\int_{\max (t_{2},0)}^{t_{1}}\mathcal{S}%
_{t_{1}s}J(s)ds.
\end{equation*}%
Combining the two formulas, we obtain 
\begin{equation*}
f(t)=\chi _{\left\{ 0<t_{2}\right\} }\mathcal{B}_{tt_{1}}\mathcal{B}%
_{t_{1}t_{2}}f_{-}(t_{2})+\mathcal{B}_{tt_{1}}\int_{\max (t_{2},0)}^{t_{1}}%
\mathcal{S}_{t_{1}s}J(s)ds+\int_{\max (t_{1},0)}^{t}\mathcal{S}_{ts}J(s)ds.
\end{equation*}
This is the case $k=2$. Iterating in the same manner by induction, that is,
replacing $f_{-}(t_{j})$ with its expression at the next level in order to
get $f_{-}(t_{j+1})$, we obtain formula (\ref{f:multiple collision}).
\end{proof}

%%%%%%%%%%%%%%%%%%%%%%%%%%%%%   SECTION   %%%%%%%%%

\section{Effect of the particle forces on the body \label%
{Sec:CharacteristicEstimates}}

The difference between the two fictitious forces \eqref{F0} and \eqref{F00}
is%
\begin{equation*}
H(t)=H\left( V(t)\right) =F_{00}(t)-F_{0}(t)=\int_{D(t)}dS_{\mathbf{x}}\int_{%
\mathbb{R}^{d}}\tilde{L}(v_{1}-V(t))h(t,\mathbf{x},\mathbf{v})d\mathbf{v}
\end{equation*}%
where we denote 
\begin{equation*}
h(t,\mathbf{x},\mathbf{v})=f_{NB}(t,\mathbf{x},\mathbf{v})-f_{0}(\mathbf{v}).
\end{equation*}%
In this expression there is no collision with the body. The function $h$
solves 
\begin{eqnarray}
\left( \partial _{t}+\mathbf{v}\cdot \nabla _{\mathbf{x}}+\mathbf{G}(t,%
\mathbf{x})\cdot \nabla _{\mathbf{v}}\right) h &=&-\mathbf{G}(t,\mathbf{x}%
)\cdot \nabla _{\mathbf{v}}f_{0}  \label{eqn:h} \\
h(0,\mathbf{x},\mathbf{v}) &=&0  \notag
\end{eqnarray}%
or written in Duhamel form, 
\begin{equation*}
h(t,\mathbf{x},\mathbf{v})=\int_{0}^{t}\mathcal{S}_{ts}J(s)ds
\end{equation*}%
where we also denote 
\begin{equation*}
J(s)=J(s,\mathbf{x},\mathbf{v})=-\mathbf{G}(s,\mathbf{x})\cdot \nabla _{%
\mathbf{v}}f_{0}.
\end{equation*}%
For any $T\in \left[ 0,t\right] $, we could also write 
\begin{equation*}
h(t,\mathbf{x},\mathbf{v})=h(T,\mathbf{\check{x}}(T;t,\mathbf{x},\mathbf{v}),%
\mathbf{\check{v}}(T;t,\mathbf{x},\mathbf{v}))-\int_{T}^{t}\mathbf{G}(s,%
\check{\mathbf{x}}(s;t,\mathbf{x},\mathbf{v}))\cdot \nabla _{\mathbf{v}%
}f_{0}(\check{\mathbf{v}}(s;t,\mathbf{x},\mathbf{v}))\ ds.
\end{equation*}%
We claim that $H\left( t\right) $ is small and integrable in time so long as 
$\mathbf{G}$ is small. To prove such a claim, we first need the following
lemma. The constants denoted as $C$ occurring in the estimates throughout
this paper are independent of $\gamma ,c_{G},t$, and $A$ as well as the
solution. The constant $A$ will be chosen sufficiently large relative to one
of the constants $C$.

%%%%%%%%% LEMMA 3.1 %%%%%%%%%%%%%%%%%

\begin{lemma}
\label{Lem:h is bounded} There is a constant $C>0$ independent of $c_{G}$, $%
\gamma $, $t$, $\mathbf{x}$\textbf{,} and $\mathbf{v}$, such that%
\begin{equation}
\left\vert h(t,\mathbf{x},\mathbf{v})\right\vert \leqslant c_{G}C\int_{0}^{t}%
\frac{1}{\left\langle s\right\rangle ^{q}}\frac{1}{\left\langle \check{x}%
_{1}(s;t,\mathbf{x},\mathbf{v})\right\rangle ^{m}}ds\ \frac{1}{\left\langle
v_{1}\right\rangle ^{l_{1}+1}}\frac{1}{\left\langle v_{\perp }\right\rangle
^{l_{2}+1}}\leq c_{G}C.  \label{estimate:h}
\end{equation}
\end{lemma}

\begin{proof}
We have 
\begin{equation*}
\mathcal{S}_{ts}J(s)=-\mathbf{G}(s,\mathbf{\check{x}}(s;t,\mathbf{x},\mathbf{%
v}))\cdot \left( \nabla _{\mathbf{v}}f_{0}\right) (\mathbf{\check{v}}(s;t,%
\mathbf{x},\mathbf{v})).
\end{equation*}%
So by assumption, 
\begin{equation*}
\left\vert \mathcal{S}_{ts}J(s)\right\vert \leqslant c_{G}C\frac{1}{%
\left\langle s\right\rangle ^{q}}\frac{1}{\left\langle \check{x}_{1}(s;t,%
\mathbf{x},\mathbf{v})\right\rangle ^{m}}\frac{1}{\left\langle \check{v}%
_{1}(s;t,\mathbf{x},\mathbf{v})\right\rangle ^{l_{1}+1}}\frac{1}{%
\left\langle \check{v}_{\perp }(s;t,\mathbf{x},\mathbf{v})\right\rangle
^{l_{2}+1}}.
\end{equation*}%
Hence 
\begin{eqnarray}
&&\left\vert h(t,\mathbf{x},\mathbf{v})\right\vert \leqslant
\int_{0}^{t}\left\vert \mathcal{S}_{ts}J(s)\right\vert ds
\label{m_estimate:h} \\
&\leqslant &c_{G}C\int_{0}^{t}\frac{1}{\left\langle s\right\rangle ^{q}}%
\frac{1}{\left\langle \check{x}_{1}(s;t,\mathbf{x},\mathbf{v})\right\rangle
^{m}}\frac{1}{\left\langle \check{v}_{1}(s;t,\mathbf{x},\mathbf{v}%
)\right\rangle ^{l_{1}+1}}\frac{1}{\left\langle \check{v}_{\perp }(s;t,%
\mathbf{x},\mathbf{v})\right\rangle ^{l_{2}+1}}ds.  \notag
\end{eqnarray}%
Finally notice that 
\begin{eqnarray}
\left\vert \check{v}_{1}(s;t,\mathbf{x},\mathbf{v})-v_{1}\right\vert
&\leqslant &\int \left\vert \mathbf{G}(p,\mathbf{\check{x}}(p;t,\mathbf{x},%
\mathbf{v}))\right\vert dp\leqslant 2\int \frac{c_{G}}{\left\langle
p\right\rangle ^{q}}dp\leqslant c_{G}C  \label{v1hat} \\
\left\vert \check{v}_{\perp }(s;t,\mathbf{x},\mathbf{v})-v_{\perp
}\right\vert &\leqslant &\int \left\vert \mathbf{G}(p,\mathbf{\check{x}}(p;t,%
\mathbf{x},\mathbf{v}))\right\vert dp\leqslant 2\int \frac{c_{G}}{%
\left\langle p\right\rangle ^{q}}dp\leqslant c_{G}C.  \notag
\end{eqnarray}%
So we can replace $\check{v}_{1}(s;t,\mathbf{x},\mathbf{v})$ and $\check{v}%
_{\perp }(s;t,\mathbf{x},\mathbf{v})$ in estimate (\ref{m_estimate:h}) by $%
v_{1}$ and $v_{\perp }$ at the price of a different constant $C$. Thereby we
obtain estimate (\ref{estimate:h}).
\end{proof}

%%%%%   LEMMA 3.2   %%%%%
The next two lemmas provide rather delicate estimates of $H(t)$.

\begin{lemma}
\label{Prop:H is small} We select $T=t^{\alpha }$ with $0<\alpha <1$ to be
determined later and assume $0<p\leq 2$ and the following (somewhat
redundant) decay conditions: 
\begin{equation}
\left( 1-\alpha \right) \left( p+1\right) >1  \label{res:H_1}
\end{equation}%
\begin{equation}
l_{1}>2\text{ and }\alpha \left( q-1\right) >1  \label{res:H_21}
\end{equation}%
\begin{equation}
\alpha m>1\text{, }q>2  \label{res:H_221}
\end{equation}%
\begin{equation}
l_{1}>q+1\text{ and }\alpha \left( q-1\right) >1.  \label{res:H_223}
\end{equation}%
Then the two fictitious forces $F_{0}$ and $F_{00}$ are near each other in
the sense that there is $C$ independent of $t$ such that%
\begin{equation*}
\left\vert H(t)\right\vert \leqslant c_{G}C\left[ \frac{1}{\left\langle
t\right\rangle ^{\left( 1-\alpha \right) \left( p+1\right) }}+\frac{1}{%
\left\langle t\right\rangle ^{\alpha m}}+\frac{1}{\left\langle
t\right\rangle ^{\alpha \left( q-1\right) }}\right] .
\end{equation*}
\end{lemma}

%%%%%%%%%%    PROOF     %%%%%%%%%%%%%%

\begin{proof}
Recall 
\begin{equation*}
H(t)=\int_{D(t)}dS_{\mathbf{x}}\int_{\mathbb{R}^{d}}\tilde{L}(v_{1}-W(t))h(t,%
\mathbf{x},\mathbf{v})d\mathbf{v}.
\end{equation*}%
First we consider $t\leq 1$. Note from \eqref{functionL} that 
\begin{equation}
\tilde{L}(v_{1}-W(t))\leq C\left( |v_{1}-W(t)|^{p}+|v_{1}-W(t)|^{2}\right)
\leq C\langle v_{1}\rangle ^{2}  \label{asymptotic:l}
\end{equation}%
because $p\leq 2$ and 
\begin{equation*}
\left\vert v_{1}-W(t)\right\vert \leq \left\vert v_{1}-V_{\infty
}\right\vert +\left\vert V_{\infty }-W(t)\right\vert \leqslant \left\vert
v_{1}\right\vert +\left\vert V_{\infty }\right\vert +\left( \gamma +A\gamma
\right) \leqslant C\langle v_{1}\rangle
\end{equation*}%
for $t\leq 1$. In the previous estimates we have used the assumption that $%
A\gamma \leqslant 1$ and $\gamma $ is small. By Lemma \ref{Lem:h is bounded}
we therefore have 
\begin{equation*}
|H(t)|\leq \int_{D(t)}dS_{\mathbf{x}}\int \langle v_{1}\rangle
^{2}Cc_{G}\langle v_{1}\rangle ^{-l_{1}-1}\langle v_{\perp }\rangle
^{-l_{2}-1}d\mathbf{v}\leq Cc_{G}
\end{equation*}%
because $l_{1}>2$ and $l_{2}>1$.

\textit{In the rest of the proof, we assume} $t>1$. We decompose $%
H(t)=H_{1}(t)+H_{2}(t)$, where%
\begin{eqnarray*}
H_{1}(t) &=&\int_{D(t)}dS_{\mathbf{x}}\int_{S}\tilde{L}(v_{1}-W(t))h(t,%
\mathbf{x},\mathbf{v})d\mathbf{v} \\
H_{2}(t) &=&\int_{D(t)}dS_{\mathbf{x}}\int_{S^{C}}\tilde{L}(v_{1}-W(t))h(t,%
\mathbf{x},\mathbf{v})d\mathbf{v}
\end{eqnarray*}%
and 
\begin{equation*}
S=\left\{ \mathbf{v\in }\mathbb{R}^{d}:\left\vert v_{1}-V_{\infty
}\right\vert \leqslant \frac{bT}{t}\right\} .
\end{equation*}%
The constant $b$ will be specified later and we choose $T=t^{\alpha }$.

The estimate for $H_{1}$ is fairly simple. In fact, %noticing that 
%\begin{equation}  \tilde{L}(v_{x}-W(t))  \le   C\left(
%|v_{x}-W(t)|^{p}+|v_{x}-W(t)|^{2}\right) ,  \label{asymptotic:l}  \end{equation}%
we have by Lemma \ref{Lem:h is bounded} that 
\begin{eqnarray*}
|H_{1}(t)| &\leqslant &c_{G}C\int_{D(t)}dS_{\mathbf{x}}\int_{\left\vert
v_{1}-V_{\infty }\right\vert \leqslant \frac{bT}{t}}dv_{1}\tilde{L}%
(v_{1}-W(t))\int dv_{\perp }\frac{1}{\left\langle v_{\perp }\right\rangle
^{l_{2}+1}} \\
&\leqslant &c_{G}C\int_{|v_{1}-V_{\infty }|\leq \frac{bT}{t}%
}\{|v_{1}-V_{\infty }|^{p}+|v_{1}-V_{\infty }|^{2}+|V_{\infty
}-W(t)|^{p}+|V_{\infty }-W(t)|^{2}\}dv_{1} \\
&\leqslant &c_{G}Cb^{p+1}\left( \frac{T}{t}\right) ^{p+1}+c_{G}C\frac{1}{%
t^{p\sigma }}b\frac{T}{t}
\end{eqnarray*}%
where the last inequality comes from $p\leq 2$ and the definition of $%
\mathcal{W}$. Note that no extra decay can be obtained from the vertical
component involving $v_{\perp }$. Now $\min (m,q-1)>\frac{1}{\alpha }$ due
to (\ref{res:H_21}) and (\ref{res:H_221}), so that by \eqref{res:H_223} we
have $\sigma >1>1-\alpha $ and 
\begin{equation*}
(p+1)(1-\alpha )<p\sigma +1-\alpha \text{, for }p>0
\end{equation*}%
that is, the second term can be absorbed into the first one: 
\begin{equation}
|H_{1}(t)|\leqslant c_{G}C\left( \frac{1}{t^{\left( 1-\alpha \right) \left(
p+1\right) }}+\frac{1}{t^{p\sigma +(1-\alpha )}}\right) \leq c_{G}C\frac{1}{%
t^{\left( 1-\alpha \right) \left( p+1\right) }}  \label{estimate:H_1}
\end{equation}%
which is integrable in $t$ due to (\ref{res:H_1}).

We now turn our focus to $H_{2}(t)$, which is split into two parts $%
H_{2}=H_{21}+H_{22}$ as follows. 
\begin{eqnarray*}
H_{21}(t) &=&\int_{D(t)}dS_{\mathbf{x}}\int_{S^{C}}\tilde{L}(v_{1}-W(t))%
\left[ \int_{T}^{t}\mathcal{S}_{ts}J(s)ds\right] d\mathbf{v} \\
H_{22}(t) &=&\int_{D(t)}dS_{\mathbf{x}}\int_{S^{C}}\tilde{L}(v_{1}-W(t))h(T,%
\mathbf{\check{x}}(T;t,\mathbf{x},\mathbf{v}),\mathbf{\check{v}}(T;t,\mathbf{%
x},\mathbf{v}))d\mathbf{v}
\end{eqnarray*}%
where $h(T,\mathbf{\check{x}}(T;t,\mathbf{x},\mathbf{v}),\mathbf{\check{v}}%
(T;t,\mathbf{x},\mathbf{v}))=\int_{0}^{T}\mathcal{S}_{ts}J(s)ds.$ Now 
\begin{eqnarray}
|H_{21}(t)| &\leqslant &c_{G}C\int_{D(t)}dS_{\mathbf{x}}\int_{S^{C}}%
\int_{T}^{t}\tilde{L}(v_{1}-W(t))\frac{1}{\left\langle s\right\rangle ^{q}}%
\frac{1}{\left\langle v_{1}\right\rangle ^{l_{1}+1}}\frac{1}{\left\langle
v_{\perp }\right\rangle ^{l_{2}+1}}dsd\mathbf{v}  \label{estimate:H_21} \\
&\leqslant &c_{G}C\int_{T}^{t}\frac{1}{s^{q}}\int_{S^{C}}C\left(
|v_{1}-W(t)|^{2}+\left\vert v_{1}-W(t)\right\vert ^{p}\right) \frac{1}{%
\left\langle v_{1}\right\rangle ^{l_{1}+1}}\frac{1}{\left\langle v_{\perp
}\right\rangle ^{l_{2}+1}}d\mathbf{v}ds  \notag \\
&\leqslant &c_{G}C\int_{T}^{t}\frac{1}{\left\langle s\right\rangle ^{q}}%
ds\leqslant c_{G}C\frac{1}{T^{q-1}}=c_{G}C\frac{1}{t^{\alpha \left(
q-1\right) }}  \notag
\end{eqnarray}%
which is integrable over $t\in (0,\infty )$ if (\ref{res:H_21}) is
satisfied. We note that the constant $C$ in $H_{21}$ depends on $V_{\infty }$%
. It remains to estimate $H_{22}$, which we put into Lemma \ref{Prop:H_22}.
\end{proof}

%%%%%%%%%%%%  LEMMA  3.3  %%%%%%%%%

\begin{lemma}
\label{Prop:H_22} Assuming the decay conditions (\ref{res:H_221}) and (\ref%
{res:H_223}), we have%
\begin{equation*}
\left\vert H_{22}(t)\right\vert \leqslant c_{G}C\left[ \frac{1}{\left\langle
t\right\rangle ^{\alpha m}}+\frac{1}{\left\langle t\right\rangle ^{\alpha
\left( q-1 \right) }}\right] \text{ and so } \int_0^\infty
|H_{22}(t)|dt<\infty.
\end{equation*}
\end{lemma}

\begin{proof}
Because the $t\leqslant 1$ case has already been taken care of in Lemma \ref%
{Prop:H is small}, \textit{we restrict ourselves to $t>1$ in this proof}.
Recall that we have chosen $T=t^{\alpha }<t$. For brevity we denote 
\begin{equation*}
\mathbf{\check{x}}=\mathbf{\check{x}}(T;t,\mathbf{x},\mathbf{v}),\quad 
\mathbf{\check{v}}=\mathbf{\check{v}}(T;t,\mathbf{x},\mathbf{v}),
\end{equation*}%
and%
\begin{equation}
\mathbf{x}^{\ast }=\mathbf{\check{x}}(s;T,\mathbf{\check{x}},\mathbf{\check{v%
}})=\mathbf{\check{x}}(s;t,\mathbf{x},\mathbf{v}),\quad \mathbf{v}^{\ast }=%
\mathbf{\check{v}}(s;T,\mathbf{\check{x}},\mathbf{\check{v}}).=\mathbf{%
\check{v}}(s;t,\mathbf{x},\mathbf{v}).  \label{x*v*}
\end{equation}%
By Lemma \ref{Lem:h is bounded} with $(t,\mathbf{x},\mathbf{v})$ replaced by 
$(T,\mathbf{\check{x}},\mathbf{\check{v}})$, we have 
\begin{eqnarray}
|H_{22}(t)| &\leqslant &c_{G}C\int_{D(t)}dS_{\mathbf{x}}\int_{S^{C}}d\mathbf{%
v}\int_{0}^{T}ds\ \tilde{L}(v_{1}-W(t))\frac{1}{\left\langle s\right\rangle
^{q}}\frac{1}{\left\langle x_{1}^{\ast }\right\rangle ^{m}}\frac{1}{%
\left\langle v_{1}\right\rangle ^{l_{1}+1}}\frac{1}{\left\langle v_{\perp
}\right\rangle ^{l_{2}+1}}  \notag \\
&\leqslant &c_{G}C\int_{D(t)}dS_{\mathbf{x}}\int_{S^{C}}dv_{1}\int_{0}^{T}ds%
\ \tilde{L}(v_{1}-W(t))\frac{1}{\left\langle s\right\rangle ^{q}}\frac{1}{%
\left\langle x_{1}^{\ast }\right\rangle ^{m}}\frac{1}{\left\langle
v_{1}\right\rangle ^{l_{1}+1}}.  \label{H22}
\end{eqnarray}%
We split $H_{22}$ into three parts $H_{22}=H_{221}+H_{222}+H_{223}$, where $%
H_{221}$ is the integral over the set 
\begin{equation*}
\{\left\vert v_{1}\right\vert <1\}\cap S^{c},
\end{equation*}%
\ $H_{222}$ is the integral over 
\begin{equation*}
\{|x_{1}^{\ast }|\geq (t/2)|v_{1}-V_{\infty }|>bT/2\}\cap \{|v_{1}|\geq 1\},
\end{equation*}%
and $H_{223}$ is the integral over 
\begin{equation*}
\{|x_{1}^{\ast }|<(t/2)|v_{1}-V_{\infty }|\}\cap \{|v_{1}|\geq 1\}\cap S^{c}.
\end{equation*}%
%
%
%
%
%
%
%
%
%
%
%
%
%
%
%
%
%
%
%
%
%
%
%
%
%
%
%The small constant $c_{2}$ will be chosen later.

We begin with $H_{221}$. Solving the characteristic equations (\ref%
{eqn:characteristic}) for $x_{1}^{\ast }$ with final data at time $T$ using %
\eqref{eq:decay of G} and \eqref{x*v*}, we have 
\begin{equation}
|x_{1}^{\ast }|\geq |\check{x}_{1}-T\check{v}_{1}+s\check{v}%
_{1}|-c_{2}(T-s)\geq |\check{x}_{1}-Tv_{1}|-T|\check{v}_{1}-v_{1}|-s|\check{v%
}_{1}|-c_{2}(T-s)  \label{x1*lowerbound}
\end{equation}%
where $c_{2}=O(c_{G})$. We estimate two of these terms as follows. On the
one hand, considering the time interval $(T,t)$ and using the characteristic
equation (\ref{eqn:characteristic}), we have 
\begin{equation*}
T\left\vert \check{v}_{1}-v_{1}\right\vert \leqslant CT\int_{T}^{t}\frac{1}{%
\left\langle p\right\rangle ^{q}}dp\leq \frac{1}{\left\langle T\right\rangle
^{q-2}}\leqslant C
\end{equation*}%
since $q>2$. On the other hand, for $t>1$ we have 
\begin{eqnarray}
|\check{x}_{1}-Tv_{1}| &\geq &t|V_{\infty }-v_{1}|-|\check{x}%
_{1}-X(t)-(T-t)v_{1}|-|X(t)-tV_{\infty }|  \notag \\
&\geq &bT-\int_{T}^{t}\int_{T}^{\tau }\frac{1}{\left\langle p\right\rangle
^{q}}dpd\tau -\int_{0}^{t}|W(\tau )-V_{\infty }|d\tau  \notag \\
&\geqslant &bT-C\frac{t}{\left\langle T\right\rangle ^{q-1}}-C(\gamma
+A\gamma )  \notag \\
&\geq &bT-C-C\frac{t}{\left\langle T\right\rangle ^{q-1}}\geqslant bT-C
\label{x1-Tv1}
\end{eqnarray}%
for some constant $C$, where we used the fact that 
\begin{equation*}
\int_{T}^{t}\left( \int_{T}^{\tau }\frac{1}{\left\langle p\right\rangle ^{q}}%
dp\right) d\tau =C_{q}\int_{T}^{t}\left( \frac{1}{\left\langle
T\right\rangle ^{q-1}}-\frac{1}{\left\langle \tau \right\rangle ^{q-1}}%
\right) d\tau \leqslant C_{q}\int_{T}^{t}\frac{d\tau }{\left\langle
T\right\rangle ^{q-1}}=C_{q}\frac{t-T}{\left\langle T\right\rangle ^{q-1}}%
\leqslant C_{q}t^{1-\alpha (q-1)}\leqslant C
\end{equation*}%
since $q>2$ and $\alpha (q-1)>1$. Therefore, combining \eqref{x1*lowerbound}
with the two previous inequalities and using $|\check{v}_{1}|\leq 1+c_{2}$,
we have 
\begin{equation*}
\left\vert x_{1}^{\ast }\right\vert \geq bT-C-c_{2}T-|\check{v}%
_{1}|s+c_{2}s\geq (2c_{2}+2)T-C-(1+c_{2})T\geq T-C
\end{equation*}%
for $0\leq s\leq T=t^{\alpha }$ if we choose $b>3c_{2}+2$. Plugging the
lower bound of $\left\vert x_{1}^{\ast }\right\vert $ given above into %
\eqref{H22}, we have 
\begin{eqnarray*}
|H_{221}(t)| &\leqslant &c_{G}C\int \int_{0}^{T}\frac{1}{\left\langle
s\right\rangle ^{q}}\frac{(\left\vert v_{1}-W(t)\right\vert ^{2}+\left\vert
v_{1}-W(t)\right\vert ^{p})}{\left\langle T\right\rangle ^{m}\left\langle
v_{1}\right\rangle ^{l_{1}+1}}dsdv_{1} \\
&\leqslant &c_{G}C\frac{1}{\left\langle T\right\rangle ^{m}}=c_{G}C\frac{1}{%
\left\langle t\right\rangle ^{\alpha m}}.
\end{eqnarray*}%
Since $\alpha m>1$, $|H_{221}(t)|$ is integrable in $t$, as desired.

Now for $H_{222}$, we have, by definition, that $|x_{1}^{\ast }|>bT/2.$
Hence 
\begin{equation*}
\left\vert H_{222}(t)\right\vert \leqslant c_{G}C\int \frac{1}{\left\langle
s\right\rangle ^{q}}ds\frac{(\left\vert v_{1}-W(t)\right\vert
^{2}+\left\vert v_{1}-W(t)\right\vert ^{p})}{\left\langle T\right\rangle
^{m}\left\langle v_{1}\right\rangle ^{l_{1}+1}}dv_{1}\leqslant c_{G}C\frac{1%
}{\left\langle T\right\rangle ^{m}}=c_{G}C\frac{1}{\left\langle
t\right\rangle ^{\alpha m}}
\end{equation*}%
which is also integrable in $t$.

Finally considering the term $H_{223}$ and using the characteristic equation
(\ref{eqn:characteristic}) in the interval $(s,T)$ as before, we have 
\begin{equation*}
\left\vert x_{1}^{\ast }\right\vert \geq |\check{x}_{1}-T\check{v}_{1}+s%
\check{v}_{1}|-c_{2}(T-s)\geq t|V_{\infty }-v_{1}|-|\check{x}%
_{1}-X(t)-(T-t)v_{1}|-|X(t)-tV_{\infty }|-s|\check{v}_{1}|-c_{2}T+c_{2}s
\end{equation*}%
where again $c_{2}=O(c_{G})$. Consideration of the interval $(T,t)$ provides
a constant bound of the second and third terms on the right side of this
inequality. In fact, we already proved in \eqref{x1-Tv1} that 
\begin{equation*}
|\check{x}_{1}-X(t)-(T-t)v_{1}|+|X(t)-tV_{\infty }|\leqslant C.
\end{equation*}%
Within the integration region $\left\{ \frac{t}{2}\left\vert v_{1}-V_{\infty
}\right\vert >|x_{1}^{\ast }|\right\} $ of $H_{223}$, we therefore have 
\begin{equation*}
\frac{t}{2}\left\vert V_{\infty }-v_{1}\right\vert >\left\vert x_{1}^{\ast
}\right\vert \geqslant t|V_{\infty }-v_{1}|-C-c_{2}T-s(|\check{v}_{1}|-c_{2})
\end{equation*}%
where $C$ is some fixed constant. Hence 
\begin{equation*}
s\geq \frac{1}{|\check{v}_{1}|-c_{2}}\left\{ \frac{t}{2}|V_{\infty
}-v_{1}|-C-c_{2}T\right\} \geq \frac{1}{|\check{v}_{1}|-c_{2}}\left\{ \left( 
\frac{b}{4}-c_{2}\right) T-C\right\} \geq \frac{c_{5}T}{|\check{v}_{1}|-c_{2}%
}\equiv s_{0},
\end{equation*}%
where we have used the definition of $S^{c}$ and where $c_{5}>0$ provided
that we choose $b$ large enough that $b>4C+4c_{2}$. Moreover, choosing $%
c_{G} $ small enough, we have by \eqref{v1hat} that 
\begin{equation*}
\left\vert \check{v}_{1}\right\vert -c_{2}\geqslant \left\vert
v_{1}\right\vert -\left\vert v_{1}-\check{v}_{1}\right\vert -c_{2}\geqslant
1-Cc_{G}-c_{2}>\tfrac{1}{2}.
\end{equation*}%
Thus from \eqref{H22} we have 
\begin{eqnarray*}
|H_{223}(t)| &\leq &c_{G}C\int_{s_{0}}^{\infty }\frac{1}{\left\langle
s\right\rangle ^{q}}ds\int_{0}^{\infty }\ \ \frac{(\left\vert
v_{1}-W(t)\right\vert ^{2}+\left\vert v_{1}-W(t)\right\vert ^{p})}{%
\left\langle v_{1}\right\rangle ^{l_{1}+1}}\ dv_{1} \\
&\leqslant &\frac{c_{G}C}{T^{q-1}}\int_{0}^{\infty }\ \frac{\langle
v_{1}\rangle ^{q-1}}{\left\langle v_{1}\right\rangle ^{l_{1}+1}}(\left\vert
v_{1}-W(t)\right\vert ^{2}+\left\vert v_{1}-W(t)\right\vert ^{p})\ dv_{1}.
\end{eqnarray*}%
Thus 
\begin{equation*}
|H_{223}(t)|\leqslant \frac{c_{G}C}{\left\langle T\right\rangle ^{q-1}}=%
\frac{c_{G}C}{\left\langle t\right\rangle ^{\alpha \left( q-1\right) }}
\end{equation*}%
because $l_{1}>q+1$. It is an integrable function of $t$ because $\alpha
(q-1)>1$.
\end{proof}

%%%%%%%%%  THEOREM 3.1   %%%%%%%%%%%%%%
We summarize the estimates of $H(t)$ in the following theorem.

\begin{theorem}
\label{Theorem:H-summary} Assuming $\mu\equiv \min (m,q-1)>\frac{p+1}{p}$, 
%   and $l_{1}>q+1 $ 
we have $\left\vert H(t)\right\vert \leqslant {c_{G}C}{\left\langle
t\right\rangle ^{-\sigma }},$ where $\sigma $ is given by \eqref{sigma}.
\end{theorem}

\begin{proof}
By Lemma \ref{Prop:H is small}, we have 
\begin{eqnarray*}
\left\vert H(t)\right\vert &\leqslant &c_{G}C\left[ \frac{1}{\left\langle
t\right\rangle ^{\left( 1-\alpha \right) \left( p+1\right) }}+\frac{1}{%
\left\langle t\right\rangle ^{\alpha m}}+\frac{1}{\left\langle
t\right\rangle ^{\alpha \left( q-1\right) }}\right] \\
&\leqslant &c_{G}C\left[ \frac{1}{\left\langle t\right\rangle ^{\left(
1-\alpha \right) \left( p+1\right) }}+\frac{1}{\left\langle t\right\rangle
^{\alpha \min (m,q-1)}}\right]
\end{eqnarray*}%
The inequality is optimized by equating the two powers. That is, we choose $%
\alpha ={(p+1)}/({p+1+\min (m,q-1)})$. Then we get the stated inequality. If 
$\min (m,q-1)>\frac{p+1}{p}$, we automatically have 
\begin{equation*}
\alpha \min (m,q-1) > 1, \quad (1-\alpha )(p+1) >1.
\end{equation*}
\end{proof}

%%%%%%%%%%%%%%%%%%%%%   SECTION PRECOLLISION   %%%%%%%%%%%%%%%%%%%%

\section{Effect of the precollisions on the body \label{Sec:recollisions}}

%The remaining force due to the precollisions by the body moving at velocity $W(t)$ is 
Because we have already estimated the term $F_{00}(W)-F_{0}(t,X,W)$ in
Theorem \ref{Theorem:H-summary}, we may now focus on the estimate of $%
R_{W}(t)$.

We begin with a couple of observations. First, recall from the earlier
discussion in Lemma \ref{IntRep} that a particle with position $\mathbf{x}$
and velocity $\mathbf{v}$ at time $t$ may have collided with the body at
various earlier times 
%$t_{1}=t_{1}(t,\mathbf{x},\mathbf{v}),\ t_{2}=t_{2}(t,\mathbf{x},\mathbf{v},%
%\mathbf{u}_{1}),\ t_{3}=t_{3}(t,\mathbf{x},\mathbf{v},\mathbf{u}_{1},\mathbf{%
%u}_{2}),\dots $ where 
$t>t_{1}>t_{2}>\dots $. %Here $\mathbf{u}_{j}$ denotes
%the velocity of the particle coming into collision with the body at position 
%$\mathbf{x}(t_{j})\in \partial \Omega (t_{j})$ and time $t_{j}$.  
Given a particle colliding at time $t$, we \textit{define the infinite
collision set} $Z(t)$ to be comprised of all points $\left( \mathbf{x},%
\mathbf{v}\right) $ for which the particle also collides with the body at a
sequence of times $s_{j}\rightarrow t$. 
%We claim that $Z(t)$ is a subset of the measure zero set 
%\begin{equation*}  \left\{ \left( \mathbf{x},\mathbf{v}\right) \ \Big|\ \mathbf{x}\in \partial
%\Omega (t_{j})\text{ and }v_{x}=W(t)\right\} .  \end{equation*}%
If $v_{1}(s)$ denotes the horizontal velocity of such a particle at time $s$%
, then 
\begin{equation*}
\int_{s_{j}}^{t}W(s)ds=\int_{s_{j}}^{t}v_{1}(s)ds
\end{equation*}
for the sequence of times $s_{j}\rightarrow t$, so that $W(t)=v_{1}(t)=v_{1}$%
. Thus such particles comprise a set of measure zero and so provide no
contribution to the force so that we may neglect them. For each of the
remaining particles the collision time $t$ is isolated.

Secondly, away from $Z(t)$, each particle that collides at time $t$ either
has a last precollision time $\tau (t,x,v)<t$ or else has no precollision at
all in the interval $(0,t)$. Accordingly, we can write the particle density
at time $t$ as the sum 
\begin{equation}
f_{-}(t,\mathbf{x},v)=\chi _{1}(t,\mathbf{v})f_{+}(\tau ,\mathbf{\check{x}}%
(\tau ;t,\mathbf{x},\mathbf{v}),\check{v}(\tau ;t,\mathbf{x},\mathbf{v}%
))+\chi _{0}(t,\mathbf{v})f_{0}(\check{v}(0;t,\mathbf{x,v}))
\label{eq:split of f-}
\end{equation}%
where the \textit{precollision characteristic functions} $\chi _{1}(t,%
\mathbf{v})$ and $\chi _{0}(t,\mathbf{v})$ are defined as follows. $\chi
_{1}(t,\mathbf{v})$ is the characteristic function of 
\begin{equation*}
\left\{ (t,\mathbf{v})\ \Big|\ \exists (\tau ,\mathbf{x})\in \left(
0,t\right) \times D(t)\text{ s.t. }\int_{\tau }^{t}W(s)ds=\int_{\tau }^{t}%
\check{v}_{x}(s;t,\mathbf{x},\mathbf{v})ds,\ \int_{\tau }^{t}\left\vert 
\check{v}_{\perp }(s;t,\mathbf{x},\mathbf{v})\right\vert ds\leqslant
2R\right\}
\end{equation*}%
with $R$ being the radius of the disk. Moreover, $\chi _{0}(t,\mathbf{v}%
)=1-\chi _{1}(t,\mathbf{v})$. The following lemma provides a crude estimate
on the velocity. %%%%%%%%%  LEMMA 4.1   %%%%%%%%%%%%%%%

\begin{lemma}
\label{Lem:recollision v range}If $\chi _{1}(t,\mathbf{v})\neq 0$, then%
\begin{equation*}
\inf_{s\leqslant t}\left\langle W\right\rangle _{s,t}-\frac{c_{G}C}{%
\left\langle t\right\rangle }\leqslant v_{1}\leqslant \sup_{s\leqslant
t}\left\langle W\right\rangle _{s,t}+\frac{c_{G}C}{\left\langle
t\right\rangle }
\end{equation*}%
where $\langle f\rangle _{s,t}$ denotes the average of $f$ over the interval 
$[s,t]$.
\end{lemma}

\begin{proof}
Taking the last collision time $\tau =\tau (t,\mathbf{x},\mathbf{v})$ before 
$t$, we have from \eqref{eqn:characteristic} 
\begin{equation*}
(t-\tau )v_{1}-c_{G}\int_{\tau }^{t}\int_{s}^{t}\frac{1}{\left\langle
p\right\rangle ^{q}}dpds\leqslant \int_{\tau }^{t}\check{v}_{1}(s;t,\mathbf{%
x,v})ds\leqslant (t-\tau )v_{1}+c_{G}\int_{\tau }^{t}\int_{s}^{t}\frac{1}{%
\left\langle p\right\rangle ^{q}}dpds
\end{equation*}%
Now 
\begin{equation*}
0\leqslant \int_{\tau }^{t}\int_{s}^{t}\frac{1}{\left\langle p\right\rangle
^{q}}dpds\leqslant C\int_{\tau }^{t}\frac{ds}{\left\langle s\right\rangle
^{q-1}}-\frac{ds}{\left\langle t\right\rangle ^{q-1}}\leqslant C\int_{\tau
}^{t}\frac{ds}{\left\langle s\right\rangle ^{q-1}}=C(t-\tau )\left\langle
\left\langle s\right\rangle ^{-q+1}\right\rangle _{\tau ,t}
\end{equation*}%
with $C$ depending on $q$. Hence 
\begin{equation*}
v_{1}(t-\tau )-c_{G}C(t-\tau )\left\langle \left\langle s\right\rangle
^{-q+1}\right\rangle _{\tau ,t}\leqslant \int_{\tau }^{t}\check{v}_{1}(s;t,%
\mathbf{x,v})ds\leqslant v_{1}(t-\tau )+c_{G}C(t-\tau )\left\langle
\left\langle s\right\rangle ^{-q+1}\right\rangle _{\tau ,t}.
\end{equation*}%
Since $\int_{\tau }^{t}\check{v}_{1}(s;t,\mathbf{x,v})ds=(t-\tau
)\left\langle W\right\rangle _{\tau ,t}$ for $\chi _{1}$, we have%
\begin{equation*}
v_{1}-c_{G}C\left\langle \left\langle s\right\rangle ^{-q+1}\right\rangle
_{\tau ,t}\leqslant \left\langle W\right\rangle _{\tau ,t}\leqslant
v_{1}+c_{G}C\left\langle \left\langle s\right\rangle ^{-q+1}\right\rangle
_{\tau ,t}\text{.}
\end{equation*}%
That is%
\begin{equation*}
\inf_{0\leqslant \tau \leqslant t}\left( \left\langle W\right\rangle _{\tau
,t}-c_{G}C\left\langle \left\langle s\right\rangle ^{-q+1}\right\rangle
_{\tau ,t}\right) \leqslant v_{1}\leqslant \sup_{0\leqslant \tau \leqslant
t}\left( \left\langle W\right\rangle _{\tau ,t}+c_{G}C\left\langle
\left\langle s\right\rangle ^{-q+1}\right\rangle _{\tau ,t}\right) .
\end{equation*}%
Using the fact that%
\begin{equation*}
\sup_{0\leqslant \tau \leqslant t}\left\langle \left\langle s\right\rangle
^{-q+1}\right\rangle _{\tau ,t}\leqslant \frac{C}{\left\langle
t\right\rangle },
\end{equation*}%
with $q>2$, we have 
\begin{equation*}
\inf_{0\leqslant \tau \leqslant t}\left\langle W\right\rangle _{\tau ,t}-%
\frac{c_{G}C}{\left\langle t\right\rangle }\leqslant v_{1}\leqslant
\sup_{0\leqslant \tau \leqslant t}\left\langle W\right\rangle _{\tau ,t}+%
\frac{c_{G}C}{\left\langle t\right\rangle }.
\end{equation*}
\end{proof}

%%%%%%%%%%%%%%%   LEMMA 4.2   %%%%%%%%%%%%%%%%
We next observe that both the collisions and the flow preserve the product
structure of the particle density.

\begin{lemma}
\label{Lem:form of f} The density $f(t,\mathbf{x},\mathbf{v})$ has the
product form $f(t,\mathbf{x},\mathbf{v})=a(t,x_{1},v_{1})\ b(t,x_{\perp
},v_{\perp })$. In fact, $b(t,x_{\perp },v_{\perp })=b_{0}(\check{v}_{\perp
}(0;t,x_{\perp },v_{\perp }))$.
\end{lemma}

\begin{proof}
Recall that $f_{0}(\mathbf{v})=a_{0}(v_{1})\ b_{0}(v_{\perp })$. Next
consider any particle $(\mathbf{x},\mathbf{v})$ at a time $s$ for which $f(s,%
\mathbf{x},\mathbf{v})=a(s,x_{1},v_{1})\ b(s,x_{\perp },v_{\perp })$. Let $%
t>s$ be times such that there is no collision in the interval $[s,t]$. Then
by \eqref{eq:simplified characteristic} we have 
\begin{eqnarray*}
f(s,\mathbf{x},\mathbf{v}) &=&f(s,\check{\mathbf{x}}(s;t,x_{1},v_{1}),\check{%
\mathbf{v}}(s;t,x_{1},v_{1})) \\
&=&a(s,\check{x}_{1}(s;t,x_{1},v_{1}),\check{v}_{1}(s;t,x_{1},v_{1}))\ b(s,%
\check{x}_{\perp }(s;t,x_{\perp },v_{\perp }),\check{v}_{\perp
}(s;t,x_{\perp },v_{\perp })).
\end{eqnarray*}

On the other hand, consider a particle that collides at time t and for which
the incoming density has the form $f_{-}(t,\mathbf{x},\mathbf{v}%
)=a(t,x_{1},v_{1})\ b(t,x_{\perp },v_{\perp })$. Then 
\begin{eqnarray*}
f_{+}(t,\mathbf{x},\mathbf{v}) &=&\int_{(u_{1}-W(t))(v_{1}-W(t))\leq
0}K(v_{1}-W(t),u_{1}-W(t))f_{-}(t,\mathbf{x},u_{1},v_{\perp })\ du_{1} \\
&=&a_{+}(t,x_{1},v_{1})\ b(t,x_{\perp },v_{\perp })
\end{eqnarray*}%
where 
\begin{equation*}
a_{+}(t,x_{1},v_{1})=\int_{(u_{1}-W(t))(v_{1}-W(t))\leq
0}K(v_{1}-W(t),u_{1}-W(t))\ a_{-}(t,x_{1},u_{1})\ du_{1}.
\end{equation*}%
Thus the product structure is preserved under both the flow and the
collisions. The last statement of the lemma is clear because the vertical
component is unaffected by the collisions.
\end{proof}

%%%%%%%%%%%%%%%   ESTIMATE ON R   %%%%%%%%%%%%%%%%

\begin{lemma}
\label{R_W estimate} As long as $\gamma $ and $c_{G}$ are small enough, we
have 
\begin{equation}
|R_{W}(t)|\leq C\frac{\left( A\gamma ^{p+1}+\gamma +c_{G}\right) ^{p+1}}{%
(1+t)^{p+1}}  \label{eq:R_W estimate}
\end{equation}%
where $C$ is independent of $t,\gamma ,c_{G}$ and $A$.
\end{lemma}

With a slightly finer proof, one can prove the better rate of decay $%
|R_{W}(t)|\leq C\frac{\left( A\gamma ^{p+1}+\gamma +c_{G}\right) ^{p+1}}{%
(1+t)^{p+d}}.$ However, by \eqref{sigma} we have $\sigma <p+1$ and hence $%
\sigma $ is a slower decay rate. So we will only bother to prove the $p+1$
decay rate for $|R_{W}(t)|$.

\begin{proof}
Without loss of generality, we may consider only the right side of the body
because the left side provides the same decay rate with the same proof. We
begin by proving an upper bound on $f_{+}$. We claim that 
\begin{equation}
b_{0}(\check{v}_{\perp }(0;t,x_{\perp },v_{\perp }))a_{+}^{\ast }\leqslant
2C_{2}b_{0}(\check{v}_{\perp }(0;t,x_{\perp },v_{\perp }))
\label{eq:upper bound for f+}
\end{equation}%
where%
\begin{equation}
a_{+}^{\ast }=\sup \left\{ a_{+}(\tau ,\xi ;u_{1})\text{ }|\text{ }\xi
=X(\tau )\text{, }\tau \in \left[ 0,\infty \right) \text{, and }u_{1}\in %
\left[ V_{\infty }-3\gamma ,V_{\infty }+3\gamma \right] \right\} .
\label{def:a+*}
\end{equation}%
Assuming $(\mathbf{x,v})\notin Z(t)$, let $t_{0}=t_{0}(t,\mathbf{x},\mathbf{v%
})$ be the first collision time after time $t$ and $\tau (t,\mathbf{x,v})$
be the last collision time before time $t$. We use (\ref{eq:split of f-}) to
split 
\begin{eqnarray}
f_{+}(t,\mathbf{x},\mathbf{v}) &=&\int_{u_{1}\leqslant
W(t)}K(v_{1}-W(t),u_{1}-W(t))\ \chi _{1}(t,u_{1},v_{\perp })
\label{eq:split with collision} \\
&&\times f_{+}(\tau ,\mathbf{\check{x}}(\tau ;t,\mathbf{x},u_{1},v_{\perp }),%
\check{v}_{1}(\tau ;t,\mathbf{x},u_{1},v_{\perp }),\check{v}_{\perp }(\tau
;t,\mathbf{x},u_{1},v_{\perp }))du_{1}  \notag \\
&&+\int_{u_{1}\leqslant W(t)}K(v_{1}-W(t),u_{1}-W(t))\ \chi
_{0}(t,u_{1},v_{\perp })  \notag \\
&&\times f_{0}(\check{v}_{1}(0;t,\mathbf{x},u_{1},v_{\perp }),\check{v}%
_{\perp }(0;t,\mathbf{x},u_{1},v_{\perp }))du_{1}  \notag \\
&=&I+II  \notag
\end{eqnarray}%
and have%
\begin{equation}
\left\vert v_{1}-\check{v}_{1}(t_{0};t,x_{1},v_{1})\right\vert \leqslant
c_{G}  \label{eq:1 good of collision}
\end{equation}%
and 
\begin{equation}
\inf_{s\leqslant t_{0}}\left\langle W\right\rangle _{s,t_{0}}-\frac{c_{G}C}{%
\left\langle t_{0}\right\rangle }\leqslant \check{v}%
_{1}(t_{0};t,x_{1},v_{1})\leqslant \sup_{s\leqslant t_{0}}\left\langle
W\right\rangle _{s,t_{0}}+\frac{c_{G}C}{\left\langle t_{0}\right\rangle }.
\label{eq:2 good of collision}
\end{equation}%
Here, (\ref{eq:1 good of collision}), which we have used many times, comes
from (\ref{eqn:characteristic}) and (\ref{eq:2 good of collision}) comes
from Lemma \ref{Lem:recollision v range}. We deduce from (\ref{eq:1 good of
collision}) and (\ref{eq:2 good of collision}) that 
\begin{equation}
\left\vert v_{1}-W(t)\right\vert \leqslant 2\left( \gamma +c_{G}\right)
<3\gamma  \label{eq:3 good of collision}
\end{equation}%
because we assumed $\gamma >2c_{G}$.

We can now estimate (\ref{eq:split with collision}). By Lemma \ref{Lem:form
of f}, 
\begin{eqnarray*}
I &=&\int_{u_{1}\leqslant W(t)}K(v_{1}-W(t),u_{1}-W(t))\ \chi
_{1}(t,u_{1},v_{\perp }) \\
&&\times a_{+}(\tau ,\check{x}(\tau ;t,x_{1},u_{1}),\check{v}_{1}(\tau
;t,x_{1},u_{1}))\ b_{0}(\check{v}_{\perp }(0;t,x_{\perp },v_{\perp }))du_{1}
\end{eqnarray*}%
and 
\begin{eqnarray*}
II &=&\int_{u_{1}\leqslant W(t)}K(v_{1}-W(t),u_{1}-W(t))\chi _{0}\
(t,u_{1},v_{\perp }) \\
&&\times a_{0}(\check{v}_{1}(0;t,x_{1},u_{1}))\ b_{0}(\check{v}_{\perp
}(0;t,x_{\perp },v_{\perp }))du_{1}.
\end{eqnarray*}%
For $I$, we use Lemma \ref{Lem:recollision v range} to deduce 
\begin{equation*}
I\leqslant b_{0}(\check{v}_{\perp }(0;t,x_{\perp },v_{\perp }))a_{+}^{\ast
}\int_{\inf_{\tau \leqslant t}\left\langle W\right\rangle _{\tau ,t}-\frac{%
c_{G}C}{\left\langle t\right\rangle }}^{W(t)}K(v_{1}-W(t),u_{1}-W(t))du_{1}.
\end{equation*}%
From the definition of $\mathcal{W}$, we deduce 
\begin{eqnarray}
W(t) &\leqslant &V_{\infty }+\frac{\left( C\gamma +A\gamma ^{p+1}\right) }{%
(1+t)^{\sigma }}\leqslant V_{\infty }+CA\gamma  \notag \\
\inf_{0<s<t}\langle W\rangle _{s,t} &\geq &V_{\infty }-\frac{1}{t}%
\int_{0}^{t}\frac{\left( C\gamma +A\gamma ^{p+1}\right) }{(1+s)^{\sigma }}%
ds\geqslant V_{\infty }-CA\gamma .  \label{estimate:infW}
\end{eqnarray}%
because $A\geqslant 1$. Hence%
\begin{equation*}
I\leqslant b_{0}(\check{v}_{\perp }(0;t,x_{\perp },v_{\perp }))C_{1}A\left(
\gamma +c_{G}\right) a_{+}^{\ast }
\end{equation*}%
because $K$ is bounded.

For the second term of (\ref{eq:split with collision}), we know as before
that $\left\vert \check{v}_{1}(0;t,X(t),u_{1})-u_{1}\right\vert <C$. Hence 
\begin{equation*}
II\leqslant b_{0}(\check{v}_{\perp }(0;t,x_{\perp },v_{\perp
}))C\int_{u_{1}\leqslant W(t)}K(v_{1}-W(t),u_{1}-W(t))\left\langle
u_{1}\right\rangle ^{-l_{1}}du_{1}
\end{equation*}%
since $a_{0}(v_{1})\leqslant {C}{\left\langle v_{1}\right\rangle ^{-l_{1}}}$%
. Noticing (\ref{eq:3 good of collision}), we use condition (\ref{KIntegral}%
) to deduce that 
\begin{equation*}
II\leqslant Cb_{0}(\check{v}_{\perp }(0;t,x_{\perp },v_{\perp }))\text{ for }%
v_{1}\in \left[ V_{\infty }-2\gamma ,V_{\infty }+2\gamma \right] .
\end{equation*}%
Thus (\ref{eq:split with collision}) becomes%
\begin{equation*}
f_{+}(t,\mathbf{x},\mathbf{v})\leqslant b_{0}(\check{v}_{\perp
}(0;t,x_{\perp },v_{\perp }))\ \{C_{1}A\left( \gamma +c_{G}\right)
a_{+}^{\ast }+C_{2}\}.
\end{equation*}%
Using Lemma \ref{Lem:form of f} on the left side of this inequality as well,
we have 
\begin{eqnarray*}
&&a_{+}(\tau ,\check{x}_{1}(\tau ;t,x_{1},u_{1}),\check{u}_{1}(\tau
;t,x_{1},u_{1}))b_{0}(\check{v}_{\perp }(0;t,x_{\perp },v_{\perp })) \\
&\leqslant &b_{0}(\check{v}_{\perp }(0;t,x_{\perp },v_{\perp }))\
\{C_{1}A\left( \gamma +c_{G}\right) a_{+}^{\ast }+C_{2}\}
\end{eqnarray*}%
so that 
\begin{eqnarray*}
b_{0}(\check{v}_{\perp }(0;t,x_{\perp },v_{\perp }))a_{+}^{\ast } &\leqslant
&b_{0}(\check{v}_{\perp }(0;t,x_{\perp },v_{\perp }))C_{1}A\left( \gamma
+c_{G}\right) a_{+}^{\ast } \\
&&+C_{2}b_{0}(\check{v}_{\perp }(0;t,x_{\perp },v_{\perp }))\text{.}
\end{eqnarray*}%
So choosing $\gamma $ so small that $A\left( \gamma +c_{G}\right)
<(2C_{1})^{-1}$, we have proven the claim (\ref{eq:upper bound for f+}).

We now estimate $R_{W}^{R}(t)$. By (\ref{eq:split of f-}), we have%
\begin{eqnarray*}
&&\left\vert f_{-}(t,\mathbf{x},\mathbf{v})-f_{NB}(t,\mathbf{x},\mathbf{v}%
)\right\vert \\
&=&\left\vert \chi _{1}(t,\mathbf{v})f_{+}(\tau ,\mathbf{\check{x}}(\tau ;t,%
\mathbf{x},\mathbf{v}),\check{v}(\tau ;t,\mathbf{x},\mathbf{v}))+\chi _{0}(t,%
\mathbf{v})f_{0}(\check{v}(0;t,\mathbf{x,v}))-f_{NB}(t,\mathbf{x},\mathbf{v}%
)\right\vert \\
&=&\chi _{1}(t,\mathbf{v})\Big|f_{+}(\tau ,\mathbf{\check{x}}(\tau ;t,%
\mathbf{x},\mathbf{v}),\check{v}(\tau ;t,\mathbf{x},\mathbf{v}))-f_{NB}(t,%
\mathbf{x},\mathbf{v})\Big| \\
&=&\chi _{1}(t,\mathbf{v})b_{0}(\check{v}_{\perp }(0;t,x_{\perp },v_{\perp
}))\Big|a_{+}(\tau ,\check{x}_{1}(\tau ;t,x_{1},v_{1}),\check{v}_{1}(\tau
;t,x_{1},v_{1}))-a_{0}(\check{v}_{1}(0;t,x_{1},v_{1}))\Big| \\
&\leqslant &C\chi _{1}(t,\mathbf{v})b_{0}(\check{v}_{\perp }(0;t,x_{\perp
},v_{\perp }))
\end{eqnarray*}%
where in the last line, we used (\ref{eq:upper bound for f+}). Hence 
\begin{eqnarray*}
\left\vert R_{W}^{R}(t)\right\vert &\leqslant
&\int_{D(t)}\int_{v_{1}\leqslant W(t)}L(v_{1}-W(t))\left\vert f_{-}(t,%
\mathbf{x},\mathbf{v})-f_{NB}(t,\mathbf{x},\mathbf{v})\right\vert d\mathbf{v}
\\
&\leqslant &C\int_{D(t)}\int_{v_{x}\leqslant W(t)}L(v_{1}-W(t))\chi _{1}(t,%
\mathbf{v})b_{0}(\check{v}_{\perp }(0;t,x_{\perp },v_{\perp }))d\mathbf{v} \\
&\leqslant &C\int_{D(t)}\int b_{0}(v_{\perp })dv_{\perp }\int_{\inf_{\tau
\leqslant t}\left\langle W\right\rangle _{\tau ,t}-\frac{c_{G}C}{%
\left\langle t\right\rangle }}^{W(t)}L(v_{1}-W(t))dv_{1} \\
&\leqslant &C\int_{\inf_{\tau \leqslant t}\left\langle W\right\rangle _{\tau
,t}-\frac{c_{G}C}{\left\langle t\right\rangle }}^{W(t)}L(v_{1}-W(t))dv_{1}
\end{eqnarray*}%
Recalling estimate (\ref{estimate:infW}), we then have by assumption on $K$
that 
\begin{eqnarray*}
\left\vert R_{W}^{R}(t)\right\vert &\leqslant &C\int_{V_{\infty }-\frac{1}{t}%
\int_{0}^{t}\eta (s)ds-\frac{c_{G}C}{\left\langle t\right\rangle }%
}^{W(t)}|v_{1}-W(t))|^{p}\ dv_{1} \\
&\leqslant &C\left\vert V_{\infty }-\frac{1}{t}\int_{0}^{t}\frac{\left(
C\gamma +A\gamma ^{p+1}\right) }{(1+s)^{\sigma }}ds-W(t)-\frac{c_{G}C}{%
\left\langle t\right\rangle }\right\vert ^{p+1} \\
&\leqslant &C\frac{\left( A\gamma ^{p+1}+\gamma +c_{G}\right) ^{p+1}}{%
(1+t)^{p+1}}
\end{eqnarray*}%
where in the last line we used the fact that 
\begin{equation*}
\frac{1}{t}\int_{0}^{t}\frac{\gamma +A\gamma ^{p+1}}{\langle s\rangle
^{\sigma }}ds\leqslant C\frac{\gamma +A\gamma ^{p+1}}{\left\langle
t\right\rangle }.
\end{equation*}
\end{proof}

%%%%%%%%%%%%%%%%%%%   Section 5   %%%%%%%%%%%%%%%%
%%%%%%%%%%%%%%%%%%%%%%%%%%%%%%%%%%

\section{Proof of Theorem \protect\ref{ThExistence}\label{Sec:existence}}

We begin with the following elementary lemma.

\begin{lemma}
\label{lem:ODE lemma}Suppose 
\begin{equation*}
\frac{dY}{dt}=-b(t)Y+d(t)
\end{equation*}%
for $t\geq 0$, where $b(t)\geq b_{0}>0$ and $\left\vert d\left( t\right)
\right\vert \leq {C_{0}}{(1+t)^{-\sigma }}$ with $\sigma >1$. Then there
exists $C_{1}$ such that 
\begin{equation*}
\left\vert Y(t)\right\vert \leqslant |Y(0)|e^{-b_{0}t}+C_{1}(1+t)^{-\sigma }
\end{equation*}%
where $C_{1}=O(C_{0})$ as $C_{0}\rightarrow 0$.
\end{lemma}

\begin{proof}
Let $B(t)=\int_{0}^{t}b(s)ds$. Then $\frac{d}{dt}[e^{B(t)}Y(t)]=e^{B(t)}d(t)$
so that 
\begin{equation*}
Y(t)=e^{-B(t)}Y(0)+e^{-B(t)}\int_{0}^{t}e^{B(s)}d(s)ds.
\end{equation*}%
Thus 
\begin{equation*}
|Y(t)|\leq |Y(0)|e^{-b_{0}t}+\int_{0}^{t}e^{-b_{0}(t-s)}C_{0}(1+s)^{-\sigma
}ds.
\end{equation*}%
Estimating $(1+s)^{-\sigma }\leq 1$ in $[0,t/2]$ and $(1+s)^{-\sigma }\leq
(1+t/2)^{-\sigma }$ in $[t/2,t]$, we find that 
\begin{equation*}
|Y(t)|\leq |Y(0)|e^{-b_{0}t}+\frac{C_{0}}{b_{0}}e^{-b_{0}t/2}+\frac{C_{0}}{%
b_{0}}\left( 1+\frac{t}{2}\right) ^{-\sigma }.
\end{equation*}
\end{proof}

\begin{lemma}
\label{coro:V_WinW} There exists a constant $A$ such that for small enough $%
\gamma $ and $c_{G}$, we have 
\begin{equation*}
\left\vert V_{W}(t)-V_{\infty }\right\vert <\gamma e^{-b_{0}t}+\frac{A\gamma
^{p+1}}{\left\langle t\right\rangle ^{\sigma }}.
\end{equation*}%
In other words, $V_{W}\in \mathcal{W}$.
\end{lemma}

\begin{proof}
We define 
\begin{equation*}
Y(t)=V_{W}(t)-V_{\infty },\quad b(t)=\frac{F_{00}(W(t))-F_{00}(V_{\infty })}{%
W(t)-V_{\infty }}
\end{equation*}%
and%
\begin{equation*}
d(t)=R_{W}(t)-F_{00}(W(t))+F_{0}(t,X(t),W(t)).
\end{equation*}%
We have $\left\vert Y(0)\right\vert =\gamma $ by definition. From Theorem %
\ref{Theorem:H-summary} and Lemma \ref{R_W estimate}, we have%
\begin{equation*}
\left\vert d(t)\right\vert \leqslant {c_{G}C}{\left\langle t\right\rangle
^{-\sigma }}+C\left( A\gamma ^{p+1}+\gamma +c_{G}\right) ^{p+1}\left\langle
t\right\rangle ^{-1-p}\newline
.
\end{equation*}%
Since $\sigma \leqslant 1+p$ by its definition (\ref{sigma}), we have%
\begin{equation*}
\left\vert d(t)\right\vert \leqslant {C\left( c_{G}+\left( A\gamma
^{p+1}+\gamma +c_{G}\right) ^{p+1}\right) \left\langle t\right\rangle
^{-\sigma }.}
\end{equation*}

Now we apply Lemma \ref{lem:ODE lemma} to obtain 
\begin{equation*}
\left\vert V_{W}(t)-V_{\infty }\right\vert \leqslant \gamma e^{-b_{0}t}+{%
C\left( c_{G}+\left( A\gamma ^{p+1}+\gamma +c_{G}\right) ^{p+1}\right)
\left\langle t\right\rangle ^{-\sigma }.}
\end{equation*}%
We choose $A>2C$ and $c_{G}<\frac{1}{2}\gamma ^{2p+1}$. We then have 
\begin{equation*}
C\left( c_{G}+\left( A\gamma ^{p+1}+\gamma +c_{G}\right) ^{p+1}\right)
<C\gamma ^{p+1}\left( \tfrac{1}{2}\gamma ^{p}+\left( A\gamma ^{p}+1+\tfrac{1%
}{2}\gamma ^{2p}\right) ^{p+1}\right) <A\gamma ^{p+1}
\end{equation*}%
by choosing $\gamma $ small enough, since $p>0$. That is,%
\begin{equation*}
\left\vert V_{W}(t)-V_{\infty }\right\vert <\gamma e^{-b_{0}t}+\frac{A\gamma
^{p+1}}{\left\langle t\right\rangle ^{\sigma }},
\end{equation*}%
as claimed.
\end{proof}

%\begin{proof}[Proof of Existence Theorem]
We now have 
\begin{equation*}
\frac{dV_{W}}{dt}=\frac{E-F_{0}(W)}{V_{\infty }-W}(V_{\infty }-V_{W})+d(t)
\end{equation*}%
with $V_{W}(0)=V_{0}$ and $\gamma =|V_{0}-V_{\infty }|$. Recall that $W\in 
\mathcal{W}$ means that $W(\cdot )$ is Lipschitz, $W(0)=V_{0}$ and $%
\left\vert W(t)-V_{\infty }\right\vert \leq \gamma e^{-b_{0}t}+\frac{A\gamma
^{p+1}}{\left\langle t\right\rangle ^{\sigma }}$. Given $L>0$, define 
\begin{equation*}
\mathcal{K}=\{W\in \mathcal{W}\ |\ esssup(|W(t)|+|\dot{W}(t)|)\leq L\}
\end{equation*}%
which is a convex and compact subset of $C_{b}([0,\infty ))$. Define the
operator $\mathcal{A}$ on $\mathcal{K}$ by $\mathcal{A}(W)=V_{W}$. A choice
of $L$ sufficiently large implies that $\mathcal{A}:\mathcal{K}\rightarrow 
\mathcal{K}$. In fact, 
\begin{eqnarray*}
\left\vert V_W(t)\right\vert &\leqslant &V_{\infty }+3\gamma \\
\left\vert \dot V_W(t)\right\vert &\leqslant &3\gamma \max_{V\in \left[
V_{\infty }-3\gamma ,V_{\infty }+3\gamma \right] }F_{00}^{\prime
}(V)+C\left( \left( \gamma +A\gamma ^{p+1}+c_{G}\right) ^{p+1}+c_{G}\right)
\end{eqnarray*}%
as already shown in the proof of Lemma \ref{coro:V_WinW}. The next lemma
will allow us to apply the Schauder fixed point theorem to deduce that $%
\mathcal{A}$ has a fixed point, which will complete the proof of existence
in Theorem \ref{ThExistence}. %\end{proof}

\begin{lemma}
If we provide $\mathcal{K}$ with the topology of $C_{b}([0,\infty ))$, then $%
\mathcal{A}$ is a continuous operator.
\end{lemma}

\begin{proof}
We follow the proof in \cite[Lemma 7.1]{CS-1}, slightly modified. The term $%
H=H_{W}=F_{00}(W)-F_{0}(t,X,W)$ involves no collisions at all, while the
recollision force satisfies $\left\vert R_{W}\right\vert \leqslant C\gamma
\langle t\rangle ^{-\sigma }\leqslant C\gamma $.

For any $t>0$, we define $B_{W}^{N}(t)$ as the set of $(\mathbf{x},\mathbf{v}%
)$ such that the trajectory passing through $(t,\mathbf{x},\mathbf{v})$ has
collided with the body at least $N+1$ times in $[0,t]$, and we define $%
A_{W}^{N}(t)$ as its complement. Using the estimate \eqref{eq:R_W estimate}
that $\sup_{t\in \lbrack 0,\infty ),\ W\in \mathcal{W}}|R_{W}(t)|\leq
C\gamma $ and iterating this estimate $N$ times, we have 
\begin{equation}
\sup_{t\in \lbrack 0,\infty ),\ W\in \mathcal{W}}\left\vert
R_{W}(t;B_{W}^{N}(t))\right\vert \leq (C\gamma )^{N}.  \label{R_W iterated}
\end{equation}%
For any $\epsilon >0$, this expression is at most $\frac{\epsilon }{4}$ by
choosing $\gamma <1/C$ and $N=N_{\epsilon }$ large enough.

Now let $W_{j}\rightarrow W$ in $\mathcal{K}$. Given any time $S>0$, we may
write 
\begin{eqnarray*}
&&\sup_{0\leq t<\infty }\left\vert R_{W_{j}}(t)-R_{W}(t)\right\vert \\
&\leq &\sup_{S\leq t<\infty }\left\vert R_{W_{j}}(t)-R_{W}(t)\right\vert
+\sup_{0\leq t<S}\left\vert R_{W}(t;B_{W}^{N}(t))\right\vert +\sup_{0\leq
t<S}\left\vert R_{W_{j}}(t;B_{W_{j}}^{N}(t))\right\vert \\
&&+\sup_{0\leq t<S}\left\vert
R_{W_{j}}(t;A_{W_{j}}^{N}(t))-R_{W}(t;A_{W}^{N}(t))\right\vert \\
&=&I+II+III+IV.
\end{eqnarray*}%
By estimate \eqref{eq:R_W estimate}, we may choose $S=S_{\varepsilon }$ so
large that $\left\vert I\right\vert <\epsilon /4.$ By estimate 
\eqref{R_W
iterated}, we have $\left\vert II\right\vert +\left\vert III\right\vert \leq
2\epsilon /4.$

Now in $IV$ there are no more than $N$ collisions. Therefore we can express
both terms in $IV$ as iterates of at most $N$ integrals by repeated use of
the collision boundary condition. The resulting \textit{finite} number of
iterated integrals contain $W_{j}$ in a \textit{finite} number of places in
the expression $R_{W_{j}}(t;A_{W_{j}}^{N}(t))$. (See Lemma \ref{IntRep}.)
Therefore they converge as $j\rightarrow \infty $ to the same expression
with $W_{j}$ replaced by $W$, and the convergence is uniform for $t\in
\lbrack 0,S]$. Thus we can choose $j$ so large that $\left\vert
IV\right\vert <\epsilon /4$. So we conclude that $R_{W_{j}}(t)\rightarrow
R_{W}(t)$ in $C_{b}([0,\infty ))$.

It is also clear that 
\begin{equation*}
H_{W_{j}}\rightarrow H_{W}
\end{equation*}%
in $\mathcal{K}$ because \textit{no collisions are involved}. It follows
from the equation \eqref{approxeqV} for $V_{W}(t)$ that $\mathcal{A}$ is
continuous in the topology of $C_{b}([0,\infty ))$.
\end{proof}

So far we have proven the existence part of Theorem \ref{ThExistence}. We
now prove that \textit{every} solution in this sense satisfies (\ref{V decay}%
) with a simple argument. In fact, consider any solution $\left( \tilde{V}%
(t),\tilde{f}\right) $ that satisfies 
\begin{equation*}
\left\vert \tilde{V}(0)-V_{\infty }\right\vert =\gamma <\gamma +A\gamma
^{p+1}.
\end{equation*}%
Then 
\begin{equation}
\left\vert \tilde{V}(t)-V_{\infty }\right\vert <\gamma e^{-b_{0}t}+\frac{%
A\gamma ^{p+1}}{\left\langle t\right\rangle ^{\sigma }}  \label{decay}
\end{equation}%
for small $t$. Now suppose we have equality at some later time. Let $T$ be
the earliest such time. 
%\begin{equation*}  \left\vert \tilde{V}(T)-V_{\infty }\right\vert =\gamma e^{-b_{0}T}+\frac{%
%A\gamma ^{p+1}}{\left\langle T\right\rangle ^{\sigma }},   \end{equation*}%
Then the existence of such a time $T$ contradicts Lemma \ref{coro:V_WinW}
because $\left( \tilde{V}(t),\tilde{f}\right) $ is a fixed point of the
mapping $\mathcal{A}$. Thus \eqref{decay} is valid for all $t<\infty $. 
\qed
%\begin{equation*}   \left\vert \tilde{V}(t)-V_{\infty }\right\vert <\gamma e^{-b_{0}t}+\frac{%
%A\gamma ^{p+1}}{\left\langle t\right\rangle ^{\sigma }}\text{ for all }t   \text{,}   \end{equation*}
%that is every solutionof the problem satisfy (\ref{V decay}).

%%%%%%%%%%%%%%%   REFERENCES   %%%%%%%%%%%%%%%%%%%%%% 

\end{document}